\documentclass[draft,12pt,a4paper,reqno]{amsart}
\usepackage{amsfonts}
\usepackage[margin=2cm]{geometry}
\usepackage{graphics, amsmath,enumitem, comment, color}
\usepackage{graphicx}
\usepackage{epsfig}
\usepackage{amssymb}
\usepackage{amsmath,amsthm}
\usepackage{mathrsfs}
\usepackage{bbm}
\usepackage{colortbl}

\newif\ifcolorcomments
\newcommand{\allowcomments}[4]{
\newcommand{#1}[1]{\ifdraft{\ifcolorcomments{\textcolor{#4}{##1 --#3}}\else{\textsl{ ##1 \ --#3}}\fi}\else{}\fi}
}
\allowcomments{\commumtaz}{MH}{Mumtaz}{green}
\allowcomments{\compb}{PB}{Phil}{blue}
\allowcomments{\comab}{AB}{A}{magenta}
\colorcommentstrue
\def\bc{\begin{center}}
\def\ec{\end{center}}
\def\be{\begin{equation}}
\def\ee{\end{equation}}

\def\N{\mathbb N}
\def\Z{\mathbb Z}

\def\R{\mathbb R}

\def\Z{\mathbb Z}

\newtheorem{lem}{Lemma}[section]

\newtheorem{pro}[lem]{Proposition}
\newtheorem{thm}[lem]{Theorem}

\newtheorem{cor}[lem]{Corollary}

\numberwithin{equation}{section}

\newtheorem*{legendre}{Legendre Theorem}
\newif\ifdraft\drafttrue

\begin{document}
\title[  Hausdorff dimension and generalized Jarn\'{i}k-Besicovitch set     ]{Hausdorff dimension for the set of points connected with the generalized Jarn\'{i}k-Besicovitch set}
\author[A. Bakhtawar]{Ayreena ~Bakhtawar}
\address{Department of Mathematics and Statistics, La Trobe University, PO
Box 199, Bendigo 3552, Australia. }
\email{A Bakhtawar: a.bakhtawar@latrobe.edu.au}
\thanks{The research of A. Bakhtawar is supported by La Trobe University postgraduate research award.}

\begin{abstract}
In this article we aim to investigate the Hausdorff dimension of the set of points $x \in [0,1)$ such that for any $r\in\mathbb{N},$
\begin{align*}  
 a_{n+1}(x)a_{n+2}(x)\cdots a_{n+r}(x)\geq e^{\tau(x)(h(x)+\cdots+h(T^{n-1}(x)))} 
\end{align*}
holds for infinitely many $n\in\mathbb{N},$ where 
 $h$ and  $\tau$ are positive continuous functions, $T$ is the Gauss map and $a_n(x)$ denote the $n$th partial quotient of $x$ in its continued fraction expansion.
By appropriate choices of $r,$ $\tau(x)$ snd $h(x)$ we obtain the classical Jarn\'{i}k-Besicovitch Theorem as well as more recent results by Wang--Wu--Xu, Wang--Wu, Huang--Wu--Xu and Hussain--Kleinbock--Wadleigh--Wang.
\end{abstract}

\maketitle

\section{Introduction}

From Dirichlet's theorem (1842) it is well known that for any irrational $x\in[0,1)$
 there exist infinitely
many pairs $(p,q)\in \Z \times \N$ such that 
\begin{equation}  \label{side1}
\left\vert x-\frac{p}{q}\right\vert<\frac{1}{q^{2}}.
\end{equation}
Statement \eqref{side1} is important as it provides the rate of approximation that works for all $x\in[0,1).$ Thus the area of interest is to investigate the sets of irrational numbers satisfying \eqref{side1} but with functions decreasing faster than one over the denominator squared. The primary metric result in this connection is due to Khintchine \cite{Khi_63} who investigated that for a decreasing function $\phi,$ the inequality
 \begin{equation}\label{side2} \left\vert x-\frac{p}{q}\right\vert<\phi(q)
 \end{equation}
  has infinitely many solutions for almost all (respectively almost no) $x\in [0,1)$ if and only if $\sum_{q=1}^\infty q\phi(q)$ diverges (respectively converges). He used the tool of continued fractions to show this.

 The metrical aspect of the theory of continued fractions has been very well studied due to its close connections with Diophantine approximation. This theory can be viewed as arising from the Gauss map $T:[0,1)\rightarrow [0,1)$ which is defined as
\begin{equation*} 
T(0)=0 \ {\rm} \ {\rm and} \ T(x)=\frac{1}{x}- \Big \lfloor \frac{1}{x} \Big \rfloor  \text{ for } \quad \text{ 0 } <x<1,
\end{equation*} 
where $\lfloor \cdot \rfloor$ denotes the integral part of any real number. Thus $T(x)$ represents the fractional part of $\frac{1}{x}.$

Also note that for any $x\in[0,1),$ its unique continued fraction expansion is given as:
\begin{equation}\label{cfrac}
x=\frac{1}{a_{1}(x)+\displaystyle{\frac{1}{a_{2}(x)+\displaystyle{\frac{1}{
a_{3}(x)+_{\ddots }}}}}}
\end{equation}
where 
for each $n\geq 1$, $a_{n}(x)$ are known as the partial quotients of $x$ such that $a_{1}(x)=\lfloor \frac{1}{x}\rfloor $ and $a_{n}(x)=\lfloor \frac{1}{T^{n-1}(x)}\rfloor$ for $n\geq 2$.  Therefore \eqref{cfrac} can also be represented as
\begin{equation*}
x=[a_{1}(x),a_{2}(x),a_{3}(x),\ldots, a_{n}(x)+T^{n}(x) ]= :[a_{1}(x),a_{2}(x),a_{3}(x),\ldots ].
\end{equation*}

Note that Khintchine's result is a Lebesgue measure criterion for the set of points satisfying equation \eqref{side2} and thus it gives no further information about the sizes of null sets. For this reason the Hausdorff dimension and measure are appropriate tools as they help to distinguish between null sets. The starting point to answer this problem is Jarn\'{i}k--Besicovitch theorem \cite{Bes34,Jar31} which gives the Hausdorff dimension of the set
\begin{equation}\label{Jset1}
J(\tau)=\left\{ x\in [ 0,1):  \left|x-\frac pq\right|  <\frac{1}{q^{\tau}}    \ \ 
\text{ for i.m. }(p,q)\in \mathbb{Z} \times \mathbb{N}\right\}.
\end{equation}
 Here and throughout `i.m.' stands for `infinitely many'.  

Recall that for any irrational $x\in [0,1),$ the irrationality exponent of $x$ is defined as
\begin{equation*}
\vartheta(x):=\sup \{\tau:x\in J(\tau)\}.
\end{equation*}
From equation \eqref{side1}, it is known that $\vartheta(x)\geq 2$ for any irrational $x\in[0,1).$
Moreover for any $\tau\geq 2,$ Jarn\'{i}k--Besicovitch theorem \cite{Bes34, Jar31} states that 
$$\dim_{\mathrm{H}}\{x\in[0,1): \vartheta(x)\geq\tau\}= \dim_{\mathrm{H}}\{x\in[0,1): \vartheta(x)=\tau\}=\frac{2}{\tau}.$$

Observe that the exponent $\tau$ in the above sets is constant. Barral and Seuret \cite{BaSe_11} generalized Jarn\'{i}k--Besicovitch theorem by considering the set of points for which the irrationality exponent is not fixed in advance but it may vary with $x$ in a continuous way. More precisely, Barral and Seuret \cite{BaSe_11} showed that for a continuous function $\tau(x),$ 

\begin{align*}
\dim_{\mathrm{H}}\{x\in[0,1): \vartheta(x)\geq \tau(x)\}&=\dim_{\mathrm{H}}\{x\in[0,1): \vartheta(x)=\tau(x)\}\\&=\frac{2}{\min \{  \tau(x) \ : \ x \in [0,1]\}}. 
\end{align*}

They called such a set the localised Jarn\'ik--Besicovitch set.
The result of Barral and Seuret was further generalised to the settings of continued fractions by Wang\textit{ et al. }in  \cite{WaWuXu16}. To refer their result, we first restate the Jarn\'{i}k--Besicovitch set \eqref{Jset1} in terms of growth rate of partial quotients,
\begin{multline}\label{Jset}
J(\tau)=\left\{ x\in [ 0,1):a_{n}(x)\geq e^{((\tau-2)/2) S_{n} (\log|T'(x)|)  } \text{ for i.m. } n\in \mathbb{N}\right\}.
\end{multline}
Note that in terms of entries of continued fractions the Jarn\'{i}k--Besicovitch set, as given in ~equation~\eqref{Jset}, contains the approximating function that involves the ergodic sum $$  S_{n} (\log|T'(x)|) =\log|T'(x)|+\cdots+ \log|T'(T^{n-1}(x))|$$ and this sum is growing fast  as $n\to\infty.$ Therefore having the approximating function in terms of the ergodic sum and the fact that partial quotients of any real number $x\in [0,1)$ completely determines its Diophantine properties, the Jarn\'{i}k--Besicovitch set (and its related variations which we will see in this article) in terms of the growth rate of partial quotients gives us better approximation results. 
In fact, Wang \textit{et al.} \cite{WaWuXu16} introduced the generalized version of the set~\eqref{Jset} as, 
 \begin{equation*}
J(\tau,h)= \left\{ x\in [ 0,1):a_{n}(x)\geq e^{\tau(x) S_{n}h(x)  }
\text{ for i.m. }n\in \mathbb{N}\right\},
\end{equation*} 
where  $h(x),$ $\tau(x)$ are positive continuous functions defined on $[0,1]$ and $S_{n}h(x)$ represents the ergodic sum `$h(x)+\cdots+h(T^{n-1}(x)).$' They called such points the localised $(\tau,h)$ approximable points. Further, they proved the Hausdorff dimension of $J(\tau,h)$ to be
\begin{equation*}s^{(1)}_\mathbb{N}=\inf \{s\geq 0 :\mathsf{P}(T, -s\tau_{\min} {h}-s\log |T^{\prime
}|)\le 0\},\end{equation*}
where $\tau_{\min}=\min \{\tau(x):x\in[0,1] \}$ and $\mathsf{P}$ denotes the pressure function defined below in Section \ref{Pres}. 

In this article we introduce the set of points $x\in[0, 1)$ for which the product of an arbitrary block of consecutive partial quotients, in their continued fraction expansion, are growing. In fact we determine the size of such a set in terms of Hausdorff dimension. 
Motivation for considering the growth of product of consecutive partial quotients arose from the work of Kleinbock and Wadleigh \cite{KlWa_19} where they considered improvements to Dirichlet's theorem. We refer the reader to \cite{BaBoHu_19, BaBoHu_20, HKWaW17, KiKi_20, KlWa_19, KlWa_2019} for comprehensive metric theory associated with the set of points improving Dirichlet's theorem.

We prove the following main result of this article.
Note that the tempered distortion property is defined in Section \ref{Pres}.

\begin{thm}
\label{ddd} Let $h: [0,1] \to (0, \infty)$ and  $\tau:[0,1]\to [0, \infty)$ be positive continuous functions with $h$ satisfying the tempered distortion property.
For $r\in \N$ define the set
\begin{align*} 
\mathcal{R}_{r}(\tau):=\left\{x\in[0, 1): a_{n+1}(x)\cdots a_{n+r}(x)\geq e^{\tau(x) S_{n}h(x)}  
\text{ for i.m. }n\in\mathbb{N}\right\}.
\end{align*}
Then
\begin{equation*} \dim_{\mathrm{H}} \mathcal R_{r}(\tau)=s^{(r)}_\N=\inf \{s\geq 0 :\mathsf{P}(T, -g_{r}(s)\tau_{\min}h-s\log |T^{\prime
}|)\le 0\},\end{equation*}
where 
$\tau_{\min}=\min\{\tau(x):x\in [0,1]\}$ and $g_r(s)$ is given by the following recursive formula $$ g_{1}(s)=s, \qquad g_{r}(s)=\frac {sg_{r-1}(s)}{1-s+g_{r-1}(s)}  \text{ for } r\geq2.$$
\end{thm}

Theorem \ref{ddd} is more general as for different $\tau(x)$ and $h(x)$ it implies various classical results as we now see.
 \begin{itemize}
\item When $r=1,$ $\tau(x)=$ constant and $h(x)=\log\vert T^{'}\vert,$ then we obtain the classical Jarn\'ik--Besicovitch theorem \cite{Bes34, Jar31}   .
\begin{cor} 
For  any $\tau\geq2,$ 
 \begin{equation*}\dim _{\mathrm{H}}J(\tau)=\frac{2}{\tau}.\end{equation*}
\end{cor}  
 \item When $r=1,$ we obtain the result by Wang \textit{et al.} \cite{WaWuXu16}.
 \begin{cor} 
\begin{equation*}
 \dim _{\mathrm{H}}    \left\{ x\in [ 0,1):a_{n+1}(x)\geq e^{\tau(x) S_nh(x)      } 
\text{ for i.m. }n\in \mathbb{N}\right\}=s^{(1)}_{\mathbb{N}}.\end{equation*}
\end{cor}
 \item When $r=1,\tau(x)=1 \ {\rm and} \ h(x)=\log B,$ we obtain the result by Wang and Wu \cite{WaWu08}.
 \begin{cor} For any $B>1,$
\begin{equation*}
 \dim _{\mathrm{H}}    \left\{ x\in [ 0,1):a_{n+1}(x)\geq B^{n}
\text{ for i.m. }n\in \mathbb{N}\right\}=s^{(1)}_{B}\end{equation*}
  where 
  $$s^{(1)}_{B}=\inf \{s\geq 0 :\mathsf{P}(T, -s\log B-s\log |T^{\prime
}|)\le 0\}.$$
\end{cor}
\item When $\tau(x)=1 \ {\rm and} \ h(x)=\log B,$ we obtain the result by 

\noindent Huang \textit{et al.}
\cite {HuWuXu}.
\begin{cor} For any $B>1,$  
 \begin{align*} 
      \dim _{\mathrm{H}}      \left\{x\in[0, 1): a_{n+1}(x)\cdots a_{n+r}(x)\geq B^{n} 
\text{ for i.m. }n\in\mathbb{N}\right\}=s^{(r)}_{B},
\end{align*}  
 where 
  $$s^{(r)}_{B}=\inf \{s\geq 0 :\mathsf{P}(T, -g_{r}(s)\log B-s\log |T^{\prime
}|)\le 0\}.$$  \end{cor}

   \item When $r=2,$ $\tau(x)={\rm constant } \text{ and } h(x)=\log\vert T^{'}\vert,$ we obtain
the result by Hussain \textit{et al.} \cite{HKWaW17}.
\begin{cor} 
 \begin{align*} 
      \dim _{\mathrm{H}}      \left\{x\in[0, 1): a_{n+1}(x)a_{n+2}(x)\geq q^{\tau +2}_{n+1}(x) 
\text{ for i.m. }n\in\mathbb{N}\right\}\\
=\frac{2}{2+\tau}.
\end{align*}  
\end{cor}
\end{itemize}

\section{Preliminaries}\label{adef}
In this section we gather some fundamental properties of continued
fractions of real numbers and recommend \cite{Khi_63, Kr_16} to the reader for further details.

For any vector ${(a_{1},\dots ,a_{n})\in \mathbb{N}^{n}}$
with $n\in \mathbb{N},$ define the \textit{basic cylinder} of order $n$ as
\begin{equation*}
I_{n}=I_{n}(a_{1},\dots ,a_{n}):=\left\{ x\in [
0,1):a_{1}(x)=a_{1},\dots ,a_{n}(x)=a_{n}\right\}.  
\end{equation*}

\begin{pro}
\label{pp3}
Let $n\geq1$ and $(a_{1},\cdots,a_{n})\in\N^{n}.$ Then

\noindent $\rm{(i)}$
\begin{equation*}
I_{n}(a_{1},a_{2},\dots ,a_{n})=\left\{ 
\begin{array}{ll}
\left[ \frac{p_{n}}{q_{n}},\frac{p_{n}+p_{n-1}}{q_{n}+q_{n-1}}\right) & 
\text{if n is even}, \\ 
\left( \frac{p_{n}+p_{n-1}}{q_{n}+q_{n-1}},\frac{p_{n}}{q_{n}}\right] & 
\text{if n is odd},
\end{array}
\right.  
\end{equation*}
 where the numerator $p_{n}=p_{n}(x)$ and the denominator $q_{n}=q_{n}(x)$ of the nth convergent of $x$ are obtained from the
following recursive relation
\begin{equation}
\begin{split}
p_{-1}& =1,~p_{0}=0,~p_{n+1}=a_{n+1}p_{n}+p_{n-1}, \\
q_{-1}& =0,~q_{0}=1,~q_{n+1}=a_{n+1}q_{n}+q_{n-1}.
\end{split}
\label{recu}
\end{equation}
Also,  $p_{n-1}q_{n}-p_{n}q_{n-1}=(-1)^{n}$ for all $n\geq1.$

\noindent $\rm{(ii)}$ The length of $I_{n}(a_{1},a_{2},\dots ,a_{n})$ is given by
\begin{equation}  \label{cyle}
\frac{1}{2q_{n}^{2}}\leq |I_{n}(a_{1},\ldots ,a_{n})|=\frac{1}{
q_{n}(q_{n}+q_{n-1})}\leq \frac{1}{q_{n}^{2}}.
\end{equation}

\noindent $\rm{(iii)}$ $q_{n}\geq 2^{(n-1)/2}$ and for
any $1\leq k \leq n,$ 
\begin{equation}\label{eq P3}
\begin{aligned}
&&q_{n+k}(a_{1},\ldots ,a_{n},a_{n+1}\ldots ,a_{n+k})\geq q_{n}(a_{1},\ldots
,a_{n})\cdot q_{k}(a_{n+1},\ldots ,a_{n+k}),   \\
&&q_{n+k}(a_{1},\ldots ,a_{n},a_{n+1}\ldots ,a_{n+k})\leq
2q_{n}(a_{1},\ldots ,a_{n})\cdot q_{k}(a_{n+1},\ldots ,a_{n+k}). 
\end{aligned}
\end{equation}
 \noindent $\rm{(iv)}$ 
\begin{equation*}
\frac{1}{(a_{n+1}+2)q_{n}^{2}}\,<\,\Big|x-\frac{p_{n}}{q_{n}}\Big|=\frac{1}{
q_n(q_{n+1}+T^{n+1}(x)q_n)}<\,\frac{1}{a_{n+1}q_{n}^{2}},
\end{equation*}
and the derivative of $T^{n}$ is given by
\begin{equation*}  
(T^{n})^{\prime }(x)=\frac{(-1)^{n}}{(xq_{n-1}-p_{n-1})^{2}}.
\end{equation*}
Further,
 \begin{equation}\label{eq1P3}
 q_{n}^{2}(x)\leq \prod^{n-1}_{k=0}|T'(T^{k}(x))|\leq 4q^{2}_{n}(x).\end{equation}
\end{pro}
From \eqref{recu} note that for any $n\geq1,$ $p_{n}$ and $q_{n}$ are determined by $
a_{1},\dots,a_{n}.$ Thus, we can write $p_{n}$ as $p_{n}(a_{1},\dots,a_{n})$ and 
$q_{n}$ as $q_{n}(a_{1},\dots,a_{n}).$ Just to avoid confusion, we can also use the notion $a_{n}$ and $q_{n}$ in place of $a_{n}(x)$
and $q_{n}(x),$ respectively.

The next theorem, known as Legendre's theorem, connects
one dimensional Diophantine approximation with continued fractions. 
\begin{legendre}
 Let $(p,q)\in \Z \times \N$. If
\begin{equation*} \Big|x-\frac pq\Big|<\frac1{2q^2} \quad {\rm then } \quad \frac pq=\frac{p_n(x)}{
q_n(x)},
\end{equation*}
$\quad \mathrm{for\ some \ } n\geq 1.$
\end{legendre}
According to Legendre's theorem if an irrational $x$ is well approximated
by a rational $\frac{p}{q}$, then this rational must be a convergent of $x$.
Thus in order to find good rational approximates to an irrational number we
only need to focus on its convergents. Note that, from $\rm{(iv)}$
of Proposition~\ref{pp3},  a real number $x$ is well approximated by its
convergent $\frac{p_{n}}{q_{n}}$ if its $(n+1)$th partial quotient is sufficiently large.

\section{The Pressure function and $s^{(r)}_\N$} \label{Pres}

In this section, we recall the definition of pressure function and collect some of its basic properties which will be useful for proving the main result of this article.
For thorough results on
pressure functions in infinite conformal iterated function systems, we refer the reader 
to \cite{ MaUr96, MaUr99, MaUr03, Wa_82}. However, for our purposes, we use the Mauldin and Urba\'{n}ski \cite{MaUr99} form of pressure function
applicable to the geometry of
continued fractions. 

Consider a finite or infinite subset $\mathcal{B}$ of the set of natural numbers and define
\begin{equation*}
Z_{\mathcal{B}}=\{x\in [0,1):a_{n}(x)\in 
\mathcal{B} \text{ for all } n\geq 1        \}.
\end{equation*}
Then $(Z_{\mathcal{B}},T)$ is a subsystem of $([0,1),T)$ where $T$ represents the 
Gauss map. Let $\psi :[0,1)\rightarrow \mathbb{R}$ be a function. Then the pressure function $\mathsf{P}_{\mathcal{B}}(T,\psi )$ with respect to potential $\psi$ and restricted to the system $(Z_{\mathcal{B}},T)$ is given by 

\begin{equation}
\mathsf{P}_{\mathcal{B}}(T,\psi):=\lim_{n\rightarrow \infty }\frac{1}{n}
\log \sum_{a_{1},\cdots ,a_{n}\in \mathcal{B}}\sup_{x\in Z_{\mathcal{B}
}}e^{S_{n}\psi ([a_{1},\cdots ,a_{n}+x])},  \label{5}
\end{equation}
where $S_{n}\psi (x)$ denotes the ergodic sum $\psi (x)+\cdots
+\psi (T^{n-1}(x))$. For $\mathcal{B}=\mathbb{N}$ we denote $\mathsf{P}_{\mathbb{N}}(T,\psi )$ by $
\mathsf{P}(T,\psi).$  

The $n$th variation of a function $\psi,$ denoted by ${\rm Var}_{n}(\psi),$ is defined as 
  \begin{equation*}
 {\rm Var}_{n}(\psi)~:~=\sup_{x,z\,  : \, I_{n}(x) = I_{n}(z)} \vert \psi(x)-\psi(z)\vert,\end{equation*}
 where $I_{n}(x)$ represents the basic cylinder of order $n$  that contains $x$ in the continued fraction expansion, that is $I_{n}(x)=I_{n}(a_{1}(x),\cdots,a_{n}(x))$ (for the definition of cylinder, see Section~\ref{adef}).
 
 A function $\psi$ is said to satisfy the \emph {tempered distortion property} if 
 \begin{equation}\label{tdp}
 {\rm Var}_{1}(\psi)<\infty \ {\rm and } \ \lim_{n\to \infty} {\rm Var}_{n}(\psi)=0.
 \end{equation}

Throughout the article, we consider $\psi :[0,1)\rightarrow \mathbb{R}$ to be a function satisfying the tempered distortion property. The existence of the limit in the definition of the pressure function is guaranteed by the following proposition.
\begin{pro}\textup{\cite[Proposition 2.4]{BiWaWuXu14}}
 The limit defining $\mathsf{P}_\mathcal{B}(T,\psi)$ exists. Furthemore if $\psi:[0,1)\to \mathbb{R}$ be a function satisfying tempered distortion property then the value of $\mathsf{P}_\mathcal{B}(T,\psi)$ remains unchanged
even without taking supremum over $x\in Z_\mathcal{B}$ in \eqref{5}.
The independence of $\mathsf{P}_\mathcal{B}(T,\psi)$ on the supremum over $x\in Z_\mathcal{B}$ follows from the fact that 
\begin{equation*}
\vert {S_{n}\psi ([a_{1},\cdots ,a_{n}+x_{1}])}-  {S_{n}\psi ([a_{1},\cdots ,a_{n}+x_{2}])}     \vert \leq \sum_{j=1}^{n}{\rm Var}_{j}(\psi)=o(n),
\end{equation*}
for any $x_{1},x_{2} \in Z_\mathcal{B}.$
\end{pro}
\noindent Thus if we want to take any point $z$ from the basic cylinder \\ $I_{n}(a_{1},\cdots,a_{n}),$ we can always take it as $z=\frac{p_{n}}{q_{n}}=[a_{1},\cdots,a_{n}].$

The next result by Hanus \textit{et al.} \cite{HaMU} shows that when
the Gauss system $([0,1),T)$ is approximated by $(Z_{\mathcal{B}},T)
$ then in the system of continued fractions the pressure function has a continuity property.
\begin{pro}\textup{\cite{HaMU}}
\label{l5} Let $\psi:[0,1)\to \mathbb{R}$ be a function satisfying \eqref{tdp}. Then 
\begin{equation*}
\mathsf{P}(T, \psi)=\sup\{\mathsf{P}_\mathcal{B}(T,\psi): 
\mathcal{B}\ \text{is  a finite subset of }\ \mathbb{N}\}.
\end{equation*}
\end{pro}

For every $
n\geq1$ and $s\geq0,$ let
\begin{equation}
f_{n, \mathcal B}\left( s \right) =\sum_{a_{1},\ldots ,a_{n}\in \mathcal{B}}\frac{1}{e^{g_{r}(s)\tau_{\min}S_{n}h(z)}q^{2s}_{n}}, \label{gnor}
\end{equation}
where $z\in I_{n}(a_{1},\cdots,a_{n})$ and $g_{r}(s)$ is defined by the formula 
\begin{equation}\label{eq6}
g_{1}(s)=s \ {\rm and}   \ g_{r}(s)=\frac {sg_{r-1}(s)}{1-s+g_{r-1}(s)} \text{ for all } r\geq2.\end{equation}
It can be easily checked that for any $s\in(\frac{1}{2},1)$ we have $g_{r+1}(s)\leq g_{r}(s),$ for all $r \geq1.$

From now onwards we consider a particular potential $$\psi_{s}(x):=-g_{r}(s)\tau_{\min}h-s\log |T^{\prime }(x)|.$$
From the definition of pressure function \eqref{5} and from \eqref{eq1P3}, \eqref{gnor} and \eqref{eq6} 
\begin{equation*}
 \mathsf P_{\mathcal{B}}(T,\psi_{s})=\lim_{n\to\infty} \frac1n \log
\sum_{a_1,\ldots,a_n \in \mathcal{B}}\frac{1}{e^{g_{r}(s)\tau_{\min}S_{n}h(z)}q^{2s}_{n}}.
\end{equation*}
Define
\begin{equation*}
s^{(r)}_{n,\mathcal B} =\inf \left\{ s \geq 0:f_{n,\mathcal B}\left( s
\right) \leq 1\right\}, 
\end{equation*} and let
\begin{align*}  
s^{(r)}_{\mathcal B} &=\inf \{s\geq 0 :\mathsf P_{\mathcal{B}}(T,-g_{r}(s)\tau_{\min}h-s\log |T^{\prime }(x)|)\le 0\}, \\
s^{(r)}_{\N} &=\inf \{s\geq 0 :          \mathsf P(T,-g_{r}(s)\tau_{\min}h-s\log |T^{\prime }(x)|)                           \le 0\}.
\end{align*}
When $\mathcal{B}$ is a finite subset of $\mathbb{N}$, then $s^{(r)}_{n,\mathcal B}$ and $s^{(r)}_\mathcal B$ are the unique solutions to $f_{n,\mathcal B}\left( s
\right)= 1 $ and $    \mathsf P_{\mathcal{B}}(T,-g_{r}(s)\tau_{\min}h-s\log |T^{\prime }(x)|)=0,$ respectively (for details see \cite{WaWu08}).
If $\mathcal B=\{1,\cdots,M\}$ for any $M\in\N,$ write $s^{(r)}_{n, M}$ for $s^{(r)}_{n,\mathcal B}$ and $s^{(r)}_{M}$ for $s^{(r)}_{\mathcal B}.$ 

From Proposition \ref{l5} and since the potential  $\psi_{s}$ satisfies the tempered distortion  property we have the following result.
\begin{cor}\label{dmno} For any integer $r\ge1,$
\begin{equation*}
s^{(r)}_{\N}=\sup \{s^{(r)}_{\mathcal B}: \mathcal B \ is \ a \ finite \ subset \ of \ \N       \}.
\end{equation*}

Furthermore, the dimensional term $s^{(r)}_{\N}$ is continuous with respect to $\psi_{s},$ that is,
\begin{equation}\label{cont}
\lim_{\epsilon\to 0}\inf \{s\geq 0 :\mathsf P_{\mathcal{B}}(T,\psi_{s}+\epsilon)\le 0\}= \inf \{s\geq 0 :\mathsf P_{\mathcal{B}}(T,\psi_{s})\le 0\}.\end{equation}

\end{cor}

From the definition of pressure function and by Corollary \ref{dmno} we have the
following result.
\begin{cor} \label{limcor}For any $M\in\N$ and $0< \epsilon <1,$
\begin{equation*}
\lim_{n\to \infty}s^{(r)}_{n,M}=s^{(r)}_M, \ \       \lim_{n\to \infty}s^{(r)}_{n,\N}=s^{(r)}_\N, \ \    \lim_{M\to \infty}s^{(r)}_M=s^{(r)}_{\N}\ {\rm and} \ \vert s^{(r)}_{n,M}-s^{(r)}_{M}\vert \leq \epsilon.
\end{equation*}
\end{cor}

\section{Proof of Theorem \protect\ref{ddd}}
The proof of Theorem \ref{ddd} is divided into two main parts: 

\noindent{\rm{(i)}} the upper estimate for $\dim_{\mathrm{H}} \mathcal{R}_{r}(\tau)$ and 

\noindent{\rm{(ii)}} the lower estimate for $\dim_{\mathrm{H}} \mathcal{R}_{r}(\tau).$

\subsection{Proof of Theorem \protect\ref{ddd}: the upper estimate.} \
Upper estimates for the Hausdorff dimension of any set are usually easy to obtain as they involve a natural covering argument. Thus for the upper bound we will find a natural covering for the set $\mathcal R_r(\tau)$.
To do this, recall that $h$ is assumed to be a positive continuous function satisfying tempered distortion property. Consequently, $$(1/n)\sum^{n}_{j=1} {\rm Var}_{j}(h)\rightarrow 0 \quad {\rm as} \ n \to \infty.$$
Thus for any fixed $\lambda>0$ there exist $N(\lambda)\in\N$ such that for any $n\geq N(\lambda)$ we have ${ \sum^{n}_{j=1}{\rm Var}_{j}(h)    \leq n \lambda}.$ Therefore, for any $x,z\in [0,1)$ with $I_{n}(x)=I_{n}(z),$  
\begin{align*}\left \vert S_{n}h(x)-S_{n}h(z) \right \vert &=  \left \vert    \sum^{n-1}_{j=0} h(T^{j}x) - \sum^{n-1}_{j=0} h(T^{j}z) \right \vert \\ &\leq \sum^{n-1}_{j=0} \left \vert    h(T^{j}x) - h(T^{j}z)  \right \vert \\ &\leq{ \sum^{n-1}_{j=0}{\rm Var}_{n-j}(h)    \leq n \lambda}.    \end{align*}
Then it follows that 
\begin{align*} 
\mathcal{R}_{r}(\tau)\subset \mathcal C_{r}(\tau)
\end{align*}
where $$\mathcal C_{r}(\tau) :=\left\{x\in[0, 1): a_{n+1}(x)\cdots a_{n+r}(x)\geq e^{\tau_{\min}S_{n}(h-\lambda)(z)} 
\text{ i. m. } n\in\mathbb{N}\right\}$$
and $z\in I_{n}(a_{1}, \cdots, a_{n}).$
Thus for the upper estimate of $\dim_{\mathrm{H}} \mathcal R_{r}(\tau)$ it is sufficient to calculate the upper estimate for
$\dim_{\mathrm{H}} \mathcal C_{r}(\tau)$ that is first we will show that
\begin{equation}  \label{drr}\dim_{\mathrm{H}} \mathcal C_{r}(\tau)\leq\inf \{s\geq 0 :\mathsf{P}(T, -g_{r}(s)\tau_{\min}(h-\lambda)-s\log |T^{\prime
}|)\le 0\},\end{equation}
by induction on $r.$

For $r=1$, the result is proved by Wang \textit{et al.} in \cite{WaWuXu16}.

Suppose that \eqref{drr} is true for $r=k.$
 We need to show that \eqref{drr} holds for $r=k+1.$
Note that
\begin{multline}
\mathcal C_{k+1}(\tau)\subseteq  \left\{x\in[0, 1): a_{n+1}(x)\cdots a_{n+k}(x)\geq e^{\tau_{\min}S_{n}(h-\lambda)(z)} 
\text{ i. m. }n\in\mathbb{N}\right\} \\ \cup  \left\{x\in[0, 1):\begin{aligned} &1\leq a_{n+1}(x)\cdots a_{n+k}(x)\leq e^{\tau_{\min}S_{n}(h-\lambda)(z)}, 
\\ & a_{n+k+1}(x)\geq \frac{e^{\tau_{\min}S_{n}(h-\lambda)(z)}}{a_{n+1}(x)\cdots a_{n+k}(x)}  \notag
\text{ i. m. }n\in\mathbb{N}  \end{aligned}   \right\}. 
\end{multline} 
Further, for any $1<\gamma \leq e$ we have
\begin{multline*}
\mathcal C_{k+1}(\tau)\subseteq  \left\{x\in[0, 1): a_{n+1}(x)\cdots a_{n+k}(x)\geq \gamma^{\tau_{\min}S_{n}(h-\lambda)(z)}
\text{ i. m. }n\in\mathbb{N}\right\}\\ \cup  \left\{x\in[0, 1):\begin{aligned}&     1 \leq a_{n+1}(x)\cdots a_{n+k}(x)\leq \gamma^{\tau_{\min}S_{n}(h-\lambda)(z)}, 
\\ & a_{n+k+1}(x)\geq \frac{e^{\tau_{\min}S_{n}(h-\lambda)(z)}}{a_{n+1}(x)\cdots a_{n+k}(x)} 
\text{ i. m. }n\in\mathbb{N}  \end{aligned}   \right\}  \\
=:\mathcal {A}(\tau) \cup \mathcal {B}(\tau).
\end{multline*}
Thus 
\begin{equation*}
\dim_{\mathrm{H}} \mathcal C_{k+1}(\tau)\leq \inf_{1< \gamma \leq         e}{\rm max} \{ \dim_{\mathrm{H}} \mathcal {A}(\tau),     \dim_{\mathrm{H}} \mathcal {B}  (\tau)    \}.
\end{equation*}
 By using induction hypothesis and since $ \gamma^{\tau_{\min}S_{n}(h-\lambda)(z)}\leq         e^{\tau_{\min}S_{n}(h-\lambda)(z)},$ 
\begin{equation*}\dim_{\mathrm{H}} \mathcal A(\tau)\leq t^{k}_{\gamma}:= \inf \{s\geq 0 :\mathsf{P}(T, -g_{k}(s) {\tau_{\min}(h-\lambda)}  \log \gamma-s\log |T^{\prime
}|)\le 0\}.      \end{equation*}
For the upper bound of $ \dim_{\mathrm{H}} \mathcal {B} (\tau)$ we proceed by finding a natural covering for this set. In terms of $\limsup$ nature of the set $\mathcal {B}(\tau)$ can be rewritten as
\begin{align*}
\mathcal {B}(\tau)&= \bigcap_{N=1}^{\infty }\bigcup_{n=N}^{\infty} \left\{x\in[0, 1):\begin{aligned}&     1 \leq a_{n+1}(x)\cdots a_{n+k}(x)\leq \gamma^{\tau_{\min}S_{n}(h-\lambda)(z)}, 
\\ & a_{n+k+1}(x)\geq \frac{e^{\tau_{\min}S_{n}(h-\lambda)(z)}}{a_{n+1}(x)\cdots a_{n+k}(x)}  \end{aligned}   \right\} 
\\& =: \bigcap_{N=1}^{\infty }\bigcup_{n=N}^{\infty}    \mathcal {B^{*}}(\tau).\end{align*}
Thus for each $n\geq N,$ the cover for $\mathcal {B^{*}}(\tau)$ will serve as a natural cover for $\mathcal {B}(\tau).$
Clearly,
\begin{align*}
\mathcal{B^{*}}(\tau)& \subseteq    \bigcup_{a_{1},\cdots, a_{n}\in\N}   \bigcup_{  {1 \leq a_{n+1}\cdots a_{n+k}\leq \gamma^{\tau_{\min}S_{n}(h-\lambda)(z)} } }  J_{n+k}(a_{1},\cdots
,a_{n+k})\end{align*}
where 
$$J_{n+k}(a_{1},\cdots
,a_{n+k})=  \bigcup_  {a_{n+k+1}\geq \frac{e^{\tau_{\min}S_{n}(h-\lambda)(z) }}{         {a_{n+1}\dots a_{n+k}}   }   }   I_{n+k+1}(a_{1},\cdots
,a_{n+k+1}).$$
By using equation \eqref{cyle},
\begin{eqnarray}
 \left\vert J_{n+k}(a_{1},\dots ,a_{n+k})\right\vert& =&\sum\limits_ {a_{n+k+1}\geq \frac{e^{\tau_{\min}S_{n}(h-\lambda)(z)}}{a_{n+1}\dots a_{n+k}}                  }           \left\vert I_{n+k+1}\left( a_{1},\dots,a_{n+k+1}\right)
\right\vert  \notag \\
&\asymp &                                                                                                                 \frac{1}{   \frac{e^{\tau_{\min}S_{n}(h-\lambda)(z)}}{a_{n+1}\cdots a_{n+k}}  q^{2}_{n+k}(a_{1},\cdots , a_{n+k})}  \notag \\ 
&\asymp & \frac{1}{        e^{\tau_{\min}S_{n}(h-\lambda)(z) }             (a_{n+1}\cdots a_{n+k})q^{2}_{n}(a_{1},\cdots , a_{n})}.\notag
\end{eqnarray}

Fixing $\delta>0$ and taking the $(s+\delta)$-volume of the cover of $B^{*}(\tau),$ we obtain
{\allowdisplaybreaks
\begin{multline*}
{} \sum\limits_    {a_{1},\cdots, a_{n}\in\N}   \sum\limits_  {1\leq a_{n+1}\cdots a_{n+k}\leq {\gamma^{\tau_{\min}S_{n}(h-\lambda)(z)}}}\\
 \qquad   \left(\frac{1}{  e^{\tau_{\min}S_{n}(h-\lambda)(z) }             (a_{n+1}\cdots a_{n+k})q^{2}_{n}(z)} \right)^{s+ \delta} \\
{}  \leq \sum\limits_    {a_{1},\cdots, a_{n}\in\N}  \sum\limits_  {1\leq a_{n+1} \cdots  {a_{n+k}}\leq {\gamma^{\tau_{\min}S_{n}(h-\lambda)(z)}}} 
  \\   \left(\frac{1}{  e^{\tau_{\min}S_{n}(h-\lambda)(z) }             (a_{n+1}\cdots a_{n+k})q^{2}_{n}(z)} \right)^{s}       e^{-\delta{\tau_{\min}S_{n}(h-\lambda)(z) }   }           \\
     \asymp   \sum\limits_    {a_{1},\cdots, a_{n}\in\N} \frac{(\log\gamma^{\tau_{\min}S_{n}(h-\lambda)(z)})^{k-1}}{(k-1)!} \gamma^{{(1-s) \tau_{\min}S_{n}(h-\lambda)(z)}} \\
  \qquad \qquad \qquad \qquad  \cdot   \left(\frac{1}{  e^{\tau_{\min}S_{n}(h-\lambda)(z) }    q^{2}_{n}(z)} \right)^{s}                e^{-\delta{\tau_{\min}S_{n}(h-\lambda)(z) }   }            \\
 \leq    \sum\limits_    {a_{1},\cdots, a_{n}\in\N} \frac{(\log e^{\tau_{\min}S_{n}(h-\lambda)(z)})^{k-1}}{(k-1)!} \gamma^{{(1-s) \tau_{\min}S_{n}(h-\lambda)(z)}}  \\
\cdot \left(\frac{1}{  e^{\tau_{\min}S_{n}(h-\lambda)(z) }    q^{2}_{n}(z)} \right)^{s}    e^{-\delta{\tau_{\min}S_{n}(h-\lambda)(z) }   }  \\
   \leq   \sum\limits_    {a_{1},\cdots, a_{n}\in\N}  \gamma^{(1-s){\tau_{\min}S_{n}(h-\lambda)(z)}} \left(\frac{1}{  e^{\tau_{\min}S_{n}(h-\lambda)(z) }    q^{2}_{n}(z)} \right)^{s}. 
\end{multline*}
}

Therefore, the $(s+\delta)$-dimensional Hausdorff measure of $\mathcal {B}(\tau)$ is
\begin{multline*}
\mathcal{H}^{s+\delta}(\mathcal{B}{(\tau)})
\leq   \liminf_{N\rightarrow \infty}\sum_{n=N}^{\infty }\sum_{a_{1},\cdots ,a_{n}\in\N}
  \sum\limits_  {1\leq a_{n+1}\cdots a_{n+k}\leq {\gamma^{\tau_{\min}S_{n}(h-\lambda)(z)}}} 
  \\
 \qquad \qquad \left(\frac{1}{  e^{\tau_{\min}S_{n}(h-\lambda)(z) }             (a_{n+1}\cdots a_{n+k})q^{2}_{n}(z)} \right)^{s+ \delta} \\  
 \leq \liminf_{N\rightarrow \infty
}\sum_{n=N}^{\infty }\sum_{a_{1},\cdots ,a_{n}\in\N} \gamma^{(1-s){\tau_{\min}S_{n}(h-\lambda)(z)}} 
\left(\frac{1}{  e^{\tau_{\min}S_{n}(h-\lambda)(z) }    q^{2}_{n}(z)} \right)^{s}.
\end{multline*}
Since $\delta>0$ is arbitrary, it follows  that 
\begin{equation*}
\dim_{\mathrm{H}} \mathcal B(\tau) \leq u^{k+1}_{\gamma} \end{equation*}
 where $u^{k+1}_{\gamma}$ is defined as 
\begin{multline*}  \inf \{s\geq 0 :\mathsf{P}(T, (1-s)\tau_{\min}(h-\lambda) \log{\gamma}-s{\tau_{\min}}(h-\lambda)-s\log |T^{\prime
}|)\le 0\}.\end{multline*}

Hence, 
$$\dim_{\mathrm{H}} \mathcal C_{k+1}(\tau)\leq \inf_{1< \gamma \leq e}{\max} \{  t^{k}_{\gamma}    , u^{k+1}_{\gamma}\}.$$
As the pressure function $P(T, \psi_{s})$ is increasing with respect to the  potential $\psi_{s}$ we know $t^{k}_{\gamma}$ is increasing and $u^{k+1}_{\gamma}$ is decreasing with respect to $\gamma$. 
Therefore the infimum is obtained at the value $\gamma$ where 
               $$-g_{k}(s) {\tau_{\min}(h-\lambda)}  \log \gamma-s\log |T^{\prime}|$$ is equal to $$(1-s)\tau_{\min}(h-\lambda) \log{\gamma}-s{\tau_{\min}}(h-\lambda)-s\log |T^{\prime}|.$$
  This  implies that     
\begin{align*}                
  &\gamma^{(1-s)\tau_{\min}(h-\lambda)}e^{-s\tau_{\min}(h-\lambda)}=\gamma^{-g_{k}(s)\tau_{\min}(h-\lambda)} \\
   \iff & -\frac{s}{(1-s)+g_{k}(s)}=  -   \log \gamma \\
 \iff & -{g_{k+1}(s)}=-g_{k}(s)\log \gamma .
\end{align*}
Hence,
\begin{equation*}  \dim_{\mathrm{H}} \mathcal C_{k+1}(\tau)\leq\inf \{s\geq 0 :\mathsf{P}(T, -g_{k+1}(s)\tau_{\min}(h-\lambda)-s\log |T^{\prime
}|)\le 0\}.\end{equation*}
Consequently, for any $r\ge1$
\begin{equation*}\dim_{\mathrm{H}} \mathcal R_{r}(\tau)\leq\inf \{s\geq 0 :\mathsf{P}(T, -g_{r}(s)\tau_{\min}(h-\lambda)-s\log |T^{\prime
}|)\le 0\}.\end{equation*}
By equation \eqref{cont} of Corollary \ref{dmno} and letting $\lambda\rightarrow 0$ we obtained the desired result.

\subsection{Proof of Theorem \protect\ref{ddd}: the lower estimate.} \
For the lower estimate of the Hausdorff dimension of $\mathcal{R}_{r}(\tau)$ we will adapt the similar method as in \cite {HuWuXu}. We first 
construct the Cantor subset $\mathcal E_{\infty}$ which sits inside the set $\mathcal{R}_{r}(\tau)$ then we distribute the measure $\mu>0$ on
$\mathcal E_{\infty}$ and obtain the Holder exponent. Lastly, we apply the mass distribution principle \cite{F_14}.

\begin{pro}[Mass Distribution Principle]
 Let $\mathcal{U}$ be a Borel subset of $\R^{d}$ and $\mu$ be a Borel measure with $\mu(\mathcal{U})>0.$ Suppose that, for some $s>0$ there exist a constant 
$c>0$ such that for any $x\in [ 0,1)$ 
\begin{equation*}
 \mu (B(x,d))\leq cd^{s},
\end{equation*}
where $B(x,d)$ denotes an open ball centred at $x$ and radius $d$. Then 
$$
\dim _{\mathrm{H}}\mathcal{U}~\geq~s.$$
\end{pro}
\noindent \textbf{Cantor subset:}

Fix $\frac{1}{2}<s<s^{(r)}_{\N}$ and chose ${1\leq \gamma_{0}\leq \gamma_{1}\leq\cdots\leq\gamma_{r-2}\leq e}$ in a way such that 
\begin{equation}\label{sun1}
\log\gamma_{i}=\frac{ g_{r}(s)(1-s)^{i}   }{    s^{i+1} }  \text { for all } 0\leq i \leq r-2. \end{equation}
\goodbreak
Moreover we have the following lemma which we will prove by induction on r .
\begin{lem} For any $r\geq 1,$
\begin{equation}\label{recrel2}g_{r}(s)=\frac{ s^{r} (2s-1)   }{    s^{r}-(1-s)^{r} }\end{equation}
satisfies the recursive relation defined in \eqref {eq6}.
\end{lem}
\begin{proof}
When $r=1,$ it is clear from \eqref{eq6} that  $$g_{1}(s)=s = \frac{ s (2s-1)   }{    s-(1-s) }.$$

Suppose \eqref{recrel2} is true for $r=k,$ then for $r=k+1$  
\begin{align*}
g_{k+1}(s)&=\frac{sg_{k}(s)}{1-s+g_{k}(s)} \text{  by \eqref{eq6} } \\
&=\frac{ s   \frac{ s^{k} (2s-1)   }{    s^{k}-(1-s)^{k} }   }{ 1-s+ \frac{ s^{k} (2s-1)   }{    s^{k}-(1-s)^{k} }   } \text { ( by induction hypothesis )}\\
&=\frac{ s^{k+1} (2s-1)   }{  s^{k} -(1-s)^{k}-s^{k+1}+s(1-s)^{k}+(2s-1)s^{k}}\\
&=\frac{ s^{k+1} (2s-1)   }{  s^{k+1} -(1-s)^{k}(1-s)}=\frac{s^{k+1}(2s-1) }{s^{k+1}-(s-1)^{k+1}}.
\end{align*}
Therefore \eqref{recrel2} is true for $r=k+1.$

\end{proof}

Thus by using \eqref{sun1} and  \eqref{recrel2} it is easy to check that the following equality holds.
\begin{equation}\label{eqpro}
\begin{aligned}
\log\gamma^{-s}_{0}=\log(\gamma^{1-s}_{0}{(\gamma_{0}\gamma_{1})}^{-s})=\cdots&=        \log((\gamma_{0}\cdots\gamma_{r-3})^{1-s}          (\gamma_{0}\cdots\gamma_{r-2})^{-s} )\\                                                               &=\log(\gamma_{0}\cdots\gamma_{r-2})^{1-s}-s ={-g_{r}(s)}.
\end{aligned}
\end{equation}

Further, let $\epsilon>0$ and $M \in \N.$ 
Fix an irrational $z_{0}$ and an integer $t_{0}$ such that for any $z\in I_{n}(z_{0})$ with $n\geq t_{0}$ we have
\begin{equation*}
\tau(z)\leq \min \{ \tau_{\min}(1+\epsilon), \tau_{\min}+\epsilon \}.
\end{equation*}
Next define  two integer sequences $\{t_{j}\}_{j\geq1}$ and  $\{m_{j}\}_{j\geq1}$ recursively where $\{m_{j}\}_{j\geq1}$ is defined to be a largely sparse integer sequence tending to infinity.
For each $j\geq1,$ define $t_{j}=t_{0}+j$  and set $n_{j}=(n_{j-1}+(r-1))+t_{j}+m_{j}+1.$

Now we construct the Cantor subset $\mathcal E_{\infty}$ level by level. We start by defining the zero level. 

\noindent \textbf{Level 0.}
Let $n_{0}+(r-1)\geq t_{2}.$
Define $$\nu^{(0,r-1)}=(a_{1}(z_{0}), a_{2}(z_{0}), \cdots,   a_{n_{0}+(r-1)}(z_{0})).$$ Then the zero level $\mathcal E_{0}$ of the Cantor set $\mathcal E_{\infty}$ is defined as 
\begin{equation*}
\mathcal E_{0}:=\mathcal F_{0}=\{I_{n_{0}+(r-1)}(\nu^{(0,r-1)}) \}.
\end{equation*}

\noindent \textbf{Level 1.}
Note that $n_{1}=(n_{0}+(r-1))+t_{1}+m_{1}+1.$ Let us define the collection of basic cylinders of order $n_{1}-1:$
\begin{equation*}
\mathcal F_{1}=\{I_{n_{1}-1}(\nu^{(0,r-1)}, \nu^{(0,r-1)}|_{t_{1}}, b^{(1)}_{1},\cdots,b^{(1)}_{m_{1}}  ) : 1\leq b_{1}^{(1)},\cdots,b^{(1)}_{m_{1}}      \leq M    \}.
\end{equation*}
For each $I_{n_{1}-1}(w^{(1)}=(\nu^{(0,r-1)}, \nu^{(0,r-1)}|_{t_{1}}, b^{(1)}_{1},\cdots,b^{(1)}_{m_{1}}))\in \mathcal F_{1}$ define the collection of sub-cylinders of order $n_{1}$:
\begin{multline}\label{aysbeh}
\mathcal E_{1,0}(w^{(1)}):= \{ I_{n_{1}}(\nu^{(1,0)})=I_{n_{1}}(w^{(1)},a_{n_{1}}): \\  \gamma_{o}^{\tau({z_{1}})S_{n_{1}-(n_{0}+(r-1))-1}h(z_{1})} \leq a_{n_{1}}<2\gamma_{0}^{\tau({z_{1}})S_{n_{1}-(n_{0}+(r-1))-1}h(z_{1})}  \}
\end{multline}
where $z_{1} \in I_{n_{1}-(n_{0}+(r-1))-1} (\nu^{(0,r-1)}|_{t_{1}}, b^{(1)}_{1},\cdots,b^{(1)}_{m_{1}}).$

Let  $ I_{n_{1}} = I_{n_{1}}(\nu^{(0,r-1)}, \nu^{(0,r-1)}|_{t_{1}}, b^{(1)}_{1},\cdots,b^{(1)}_{m_{1}}, a_{n_{1}}  )  \in  \mathcal E_{1,0}(w^{(1)}).  $      The choice of $z_{1}$ indicates that for any $x\in I_{n_{1}}      $ the continued fraction representations of $z_{1}$ and $x$ share prefixes up to $t_{1}$th partial quotients. Hence $\tau(z_{1})$ is close to $\tau(x)$ by the continuity of  $\tau.$ Further, it can be easily checked that $S_{n_{1}-(n_{0}+(r-1))-1}h(z_{1})~\sim~S_{n_{1}-1} h(x),$ ( here `$\sim$' denotes the asymptotic equality of two functions). Consequently, 
\begin{align*}
&&{\tau({z_{1}})S_{n_{1}-(n_{0}+(r-1))-1}h(z_{1})} \sim \tau(x)S_{n_{1}-1} h(x)\\ 
&\implies&{\tau({z_{1}})S_{n_{1}-(n_{0}+(r-1))-1}h(z_{1})} \log\gamma_{0}\sim \tau(x)S_{n_{1}-1} h(x)\log\gamma_{0}\\
 &\implies& \log\gamma^{{\tau({z_{1}})S_{n_{1}-(n_{0}+(r-1))-1}h(z_{1})}}_{0}\sim \log\gamma^{\tau(x)S_{n_{1}-1} h(x)}_{0}\\
&\implies&   \log a_{n_{1}}\sim \log\gamma^{\tau(x)S_{n_{1}-1} h(x)}_{0}  \text { from \eqref{aysbeh} }. \\
\end{align*}
Thus,
$a_{n_{1}}(x)\sim \gamma_{o}^{\tau(x)S_{n_{1}-1}h(x)}.$

Next for each $I_{n_{1}}(\nu^{(1,0)})\in \mathcal E_{1,0}(w^{(1)})$ define 
\begin{multline*}
\mathcal E_{1,1}(\nu^{(1,0)}):= \{ I_{n_{1}+1}(\nu^{(1,1)}) =I_{n_{1}+1}(\nu^{(1,0)},a_{n_{1}+1}): \\ \gamma_{1}^{\tau({z_{1}})S_{n_{1}-(n_{0}+(r-1))-1}h(z_{1})} \leq a_{n_{1}+1} <2\gamma_{1}^{\tau({z_{1}})S_{n_{1}-(n_{0}+(r-1))-1}h(z_{1})}    \}.
\end{multline*}

Continuing in this way for each $I_{n_{1}+(r-3)}(\nu^{(1,r-3)})\in \mathcal E_{1,r-3}(\nu^{(1,r-4)})$ collect a family of sub-cylinder of order $n_{{1}+(r-2)}$:
\begin{multline*}
\mathcal E_{1,r-2}(\nu^{(1,r-3)}):= \{ I_{n_{1}+(r-2)}(\nu^{(1,r-2)}) =I_{n_{1}+(r-2)}(\nu^{(1,r-3)},a_{n_{1}+(r-2)}): 
\\  \gamma_{r-2}^{\tau({z_{1}})S_{n_{1}-(n_{0}+(r-1))-1}h(z_{1})} \leq a_{n_{1}+(r-2)} <2\gamma_{r-2}^{\tau({z_{1}})S_{n_{1}-(n_{0}+(r-1))-1}h(z_{1})} \}.
\end{multline*}

Further for each $I_{n_{1}+(r-2)}(\nu^{(1,r-2)})\in \mathcal E_{1,r-2}(\nu^{(1,r-3)})$ collect a family of sub-cylinders of order $n_{{1}+(r-1)}$:
\begin{multline*}
\mathcal E_{1,r-1}(\nu^{(1,r-2)}):= \{ I_{n_{1}+(r-1)}(\nu^{(1,r-1)})=I_ {n_{1}+(r-1)}( \nu^{(1,r-2)},a_{n_{1}+(r-1)}): 
\\  \left ( \frac{e}{\gamma_{0}\cdots\gamma_{r-2}}\right )          ^{\tau({z_{1}})S_{n_{1}-(n_{0}+(r-1))-1}h(z_{1})}  \leq   a_{n_{1}+(r-1)} 
\\<  2\left ( \frac{e}{\gamma_{0}\cdots\gamma_{r-2}}\right )^{\tau({z_{1}})S_{n_{1}-(n_{0}+(r-1))-1}h(z_{1})} \}. 
\end{multline*}

Then the first level of the Cantor set $\mathcal E_{\infty}$ is defined as 
\begin{multline*}
\mathcal E_{1,r-1}= \{   I_{n_{1}+(r-1)} (\nu^{(1,r-1)}) \in    \mathcal E_{1,r-1}(\nu^{(1,r-2)}): \\ I_{n_{1}+i}(\nu^{(1,i)})   \in      \mathcal E_{1,i}(\nu^{(1,i-1)}) \ {\rm for}\ 1\leq i\leq r-2 ; \\      I_{n_{1}}(\nu^{(1,0)} ) \in \mathcal E_{1,0}(w^{(1)}); I_{n_{1}-1}(w^{(1)})\in \mathcal F_{1}   \}.
\end{multline*}

\noindent \textbf{Level j.}

Suppose that $\mathcal E_{j-1,r-1},$ i.e., the $(j-1)$th level has been constructed. The set $\mathcal E_{j-1,r-1}$  consist of cylinders all of which are of order $n_{{j-1}}+(r-1).$
 Recall that $n_{j}~=(n_{j-1}+(r-1))+t_{j}+m_{j}+1.$ For each basic cylinder $I_{n_{j-1}+(r-1)} (\nu^{(j-1,r-1)}) \in    \mathcal E_{j-1,r-1}$ define the collections of sub-cylinders of order $n_{j}-1:$
 \begin{align*}
\mathcal F_{j}(I_{n_{j-1}+(r-1)} (\nu^{(j-1,r-1)}))=\{I_{n_{j}-1}&(\nu^{(j-1,r-1)}, \nu^{(j-1,r-1)}|_{t_{j}}, b^{(j)}_{1},\cdots,b^{(j)}_{m_{j}}): \\ &1\leq b^{(j)}_{1},\cdots, b^{(j)}_{m_{j}}\leq M\}
\end{align*} 
  and let
 \begin{equation*}
 \mathcal F_{j}= \bigcup_{I_{n_{j-1}+(r-1)}\in     \mathcal E_{j-1,r-1}} \mathcal F_{j}(I_{n_{j-1}+(r-1)} (\nu^{(j-1,r-1)})). 
 \end{equation*}
 
 Following the same process as for \textbf{Level 1}, for each
 $I_{n_{j}-1}(w^{(j)})\in \mathcal F_{j}$ define the collection of sub-cylinders:
\begin{multline*}
\mathcal E_{j,0}(w^{(j)}):=  \{ I_{n_{j}}(\nu^{(j,0)})=I_{n_{j}}(w^{(j)},a_{n_{j}}): \\ \gamma_{0}^{\tau({z_{j}})S_{n_{j}-(n_{j-1}+(r-1))-1}h(z_{j})} \leq a_{n_{j}}<2\gamma_{0}^{\tau({z_{j}})S_{n_{j}-(n_{j-1}+(r-1))-1}h(z_{j})}  \}
\end{multline*}
where $z_{j} \in I_{n_{j}-(n_{j-1}+(r-1))-1} (\nu^{(j-1,r-1)}|_{t_{j}}, b^{(j)}_{1},\cdots,b^{(j)}_{m_{j}}).$
 
Next for each $I_{n_{j}}(\nu^{(j,0)})\in \mathcal E_{j,0}(w^{(j)}) $ define
\begin{multline*}
\mathcal E_{j,1}(\nu^{(j,0)}):=  \{ I_{n_{j}+1}(\nu^{(j,1)})=I_{n_{j}+1}(\nu^{(j,0)},a_{n_{j}+1}): \\ \gamma_{1}^{\tau({z_{j}})S_{n_{j}-(n_{j-1}+(r-1))-1}h(z_{j})}  \leq a_{n_{j}+1} <2\gamma_{1}^{\tau({z_{j}})S_{n_{j}-(n_{j-1}+(r-1))-1}h(z_{j})}    \}.
\end{multline*}
 Similarly for each $I_{n_{j}+i-1}(\nu^{(j,i-1)})\in \mathcal E_{j,i-1}(\nu^{(j,i-2)})$ with $2\leq i \leq r-2$ we
collect a family of the sub-cylinders of order $n_{{j}+i}$,
 \begin{multline*}
\mathcal E_{j,i}(\nu^{(j,i-1)}):= \{ I_{n_{j}+i}(\nu^{(j,i)}) =I_{n_{j}+i}(\nu^{(j,i-1)},a_{n_{j}+i}): 
\\  \gamma_{i}^{\tau({z_{j}})S_{n_{j}-(n_{j-1}+(r-1))-1}h(z_{j})} \leq a_{n_{j}+i} <2\gamma_{i}^{\tau({z_{j}})S_{n_{j}-(n_{j-1}+(r-1))-1}h(z_{j})} \}.
\end{multline*}
  Continuing in this way for each $I_{n_{j}+r-2}(\nu^{(j,r-2)})\in \mathcal E_{j,r-2}(\nu^{(j,r-3)})$ we define 
 \begin{multline*}
\mathcal E_{j,r-1}(\nu^{(j,r-2)}):= \{I_{n_{j}+(r-1)}(\nu^{(j,r-1)})=I_{n_{j}+(r-1)}(\nu^{(j,r-2)},a_{n_{j}+(r-1)})): \\
 \left ( \frac{e}{\gamma_{0}\gamma_{1}\cdots\gamma_{r-2}}\right )^{\tau({z_{j}})S_{n_{j}-(n_{j-1}+(r-1))-1}h(z_{j})}\\   \leq a_{n_{j}+(r-1)}<  2\left ( \frac{e}{\gamma_{0}\gamma_{1}\cdots\gamma_{r-2}}\right )^{\tau({z_{j}})S_{n_{j}-(n_{j-1}+(r-1))-1}h(z_{j})}    \}.
\end{multline*}
Then the $j$th level of the Cantor set $\mathcal E_{\infty}$ is defined as 
\begin{multline*}
\mathcal E_{j,r-1}= \{   I_{n_{j}+(r-1)} (\nu^{(j,r-1)}) \in    \mathcal E_{j,r-1}(\nu^{(j,r-2)}):  \\ I_{n_{j}+i}(\nu^{(j,i)})   \in      \mathcal E_{j,i}(\nu^{(j,i-1)}) \ {\rm for}\ 1\leq i\leq r-2 ; \\     I_{n_{j}}(\nu^{(j,0)} ) \in \mathcal E_{j,0}(w^{(j)}); I_{n_{j}-1}(w^{(j)})\in \mathcal F_{j}   \}.
\end{multline*}
 
 Then the Cantor set is defined as 
 \begin{equation*}
  \mathcal E_{\infty}=\bigcap_{j=1}^{\infty }\bigcup_{I_{n_{j}+(r-1)}(\nu^{(j,r-1)}) \in \mathcal E_{j,r-1}}I_{n_{j}+(r-1)}(\nu^{(j,r-1)}).
 \end{equation*}
 
 By the same arguments as discussed earlier in defining Level 1, that is, by the continuity of~$\tau$ and since $z_{j}$ and $x$ share common prefixes up to $t_{j}$th partial quotients, we have
 \begin{equation}\label{sub}\lim_{j\to\infty}\tau(z_{j})=  \tau(x) \ {\rm and} \ \lim_{j\to\infty}\frac{S_{n_{j}-({n_{j-1}+(r-1))} -1}h(z_{j})}{S_{n_{j}-1}h(x)} =1.        \end{equation}
 Therefore $  \mathcal E_{\infty} $ is contained in $\mathcal R_{r}(\tau),$ as for showing this it is sufficient to show that \eqref{sub} holds.  
 
 In order to better understand the structure of $\mathcal E_{\infty}$ we will utilize the idea of symbolic space. If  the continued fraction expansion of a point $x \in \mathcal E_{\infty}$ is represented by $[\nu_{1},\nu_{2},\cdots,\nu_{n},\cdots]$  then the sequence $(\nu_{1},\nu_{2},\cdots,\nu_{n},\cdots)$ is known as \textit{admissible sequence} and      $\nu=(\nu_{1},\nu_{2},\cdots,\nu_{n})$ is called an \textit{admissible block} for any $n\geq1.$ If $\nu$ is an admissible block only than $I_{n}(\nu) \cap \mathcal E_{\infty}\neq \emptyset,$ and such basic cylinders $I_{n}(\nu)$ are known as admissible cylinders.
 
 For any $n\geq1,$ denote by `$\mathcal D_{n}$' the set of strings defined as
 \begin{equation*}
 \mathcal D_{n}=\{(\nu_{1},\nu_{2},\cdots,\nu_{n})\in {\mathbb N}^{n}: \nu=(\nu_{1},\nu_{2},\cdots,\nu_{n}) \text{ is an admissible block }\}.
  \end{equation*}
  We will define $\mathcal D_{n}$ for different cases, according to the limitations on the partial quotients defined in the construction of $\mathcal E_{\infty}.$
 
 \noindent Let $l_{1}=n_{1}-(n_{0}+(r-1))-1=t_{1}+m_{1}.$
 \begin{itemize}
 \item [\rm(1a)] When $1\leq n \leq (n_{0}+(r-1)),$
 $$\mathcal D_{n}=\{ \nu^{(0,r-1)}=(a_{1}(z_{0}), a_{2}(z_{0}), \cdots,   a_{n_{0}+(r-1)}(z_{0}))\}.$$ 
 
 \item [\rm(1b)] When $(n_{0}+(r-1))<n\leq (n_{0}+(r-1))+t_{1},$
  $$\mathcal D_{n}=\{ (\nu^{(0,r-1)}, \nu^{(0,r-1)}|_{{n-(n_{0}+(r-1))}}   )\}.$$ 
   
  \item [\rm(1c)] When $(n_{0}+(r-1))+t_{1}<n< n_{1},$
    \begin{multline*}\mathcal D_{n}=\{ \nu=(\nu^{(0,r-1)}, \nu^{(0,r-1)}|_{t_{1}},  \nu_{(n_{0}+(r-1))+t_{1}+1)},\dots,\nu_{n} ):\\ 1\leq\nu_{u}\leq M,\ (n_{0}+(r-1))+t_{1}<u\leq n \}.
  \end{multline*}  
    
   \item [\rm(1d)] When $n=n_{1},$  
\begin{multline*} \mathcal D_{n}=  \{ \nu=(\nu^{(0,r-1)}, \nu^{(0,r-1)}|_{t_{1}}, \nu_{(n_{0}+(r-1))+t_{1}+1)},\dots,\nu_{n_{1}-1}, \nu_{n_{1}} ): \\ \gamma_{0}^{\tau({z_{1}})S_{l_{1}}h(z_{1})} \leq \nu_{n_{1}}<2\gamma_{0}^{\tau({z_{1}})S_{l_{1}}h(z_{1})}  \\
 {\rm and} \  1\leq\nu_{u}\leq M,\ {\rm for}\ (n_{0}+(r-1))+t_{1}<u< n_{1}   \},
  \end{multline*}
   where 
 $z_{1} \in I_{l_{1} {(\nu^{(0,r-1)}|_{t_{1}},  \nu_{(n_{0}+(r-1))+t_{1}+1} ,\cdots, \nu_{n_{1}-1} )}}.$   
 
 \item [\rm(1e)] When $n=n_{1}+i$ where $1\leq i \leq r-2,$  
  \begin{multline*} \mathcal D_{n}=  \{ \nu=(\nu^{(0,r-1)}, \nu^{(0,r-1)}|_{t_{1}}, \nu_{(n_{0}+(r-1))+t_{1}+1)},\dots, \nu_{n_{1}+i} ):\\ \gamma_{i}^{\tau({z_{1}})S_{l_{1}}h(z_{1})} \leq \nu_{n_{1}+i}<2 \gamma_{i}^{\tau({z_{1}})S_{l_{1}}h(z_{1})} \ {\rm where } \ 1\leq i \leq r-2 , \\
     \gamma_{0}^{\tau({z_{1}})S_{l_{1}}h(z_{1})} \leq \nu_{n_{1}}<2 \gamma_{0}^{\tau({z_{1}})S_{l_{1}}h(z_{1})}    \\ {\rm and} \  1\leq\nu_{u}\leq M,\ {\rm for}\ (n_{0}+(r-1))+t_{1}<u<n_{1}  \}.  
  \end{multline*}
  
  \item [\rm(1f)] When $n=n_{1}+(r-1),$  
    \begin{multline*} \mathcal D_{n}=  \{ \nu=(\nu^{(0,r-1)}, \nu^{(0,r-1)}|_{t_{1}},  \nu_{(n_{0}+(r-1))+t_{1}+1)},\dots, \nu_{n_{1}} ): \\   \left ( \frac{e}{\gamma_{0}\gamma_{1}\cdots\gamma_{r-2}}\right )          ^{\tau({z_{1}})S_{{l_{1}}}h(z_{1})}  \leq  \nu_{n_{1}+(r-1)} \\  \qquad \qquad \qquad \qquad \qquad \qquad <   2\left ( \frac{e}{\gamma_{0}\gamma_{1}\cdots\gamma_{r-2}}\right )^{\tau({z_{1}})S_{l_{1}}h(z_{1})},  
 \\ \gamma_{i}^{\tau({z_{1}})S_{l_{1}}h(z_{1})} \leq \nu_{n_{1}+i}<2 \gamma_{i}^{\tau({z_{1}})S_{l_{1}}h(z_{1})}  {\rm where} \  1\leq i \leq r-2 , 
  \\   \gamma_{0}^{\tau({z_{1}})S_{l_{1}}f(z_{1})} \leq \nu_{n_{1}}<2 \gamma_{0}^{\tau({z_{1}})S_{l_{1}}h(z_{1})} \\ {\rm and} \ 1\leq\nu_{u}\leq M,\ {\rm for}\ (n_{0}+(r-1))+t_{1}<u< n_{1}  \}.  
  \end{multline*}
  
   \end{itemize}
 
 Next we define $\mathcal D_{n}$ inductively. For this we suppose that  $\mathcal D_{n_{j-1}+(r-1)}$ has been given and  for each $j\geq1,$ write $l_{j}=n_{j}-(n_{j-1}+(r-1))-1=t_{j}+m_{j}.$ 
  \begin{itemize}
   \item [\rm(2a)] When $(n_{j-1}+(r-1))<n\leq (n_{j-1}+(r-1))+t_{j},$
  $$\mathcal D_{n}=\{ \nu=(\nu^{(j-1,r-1)}, \nu^{(j-1,r-1)}|_{{n-(n_{j-1}+(r-1))}}   )\}.$$ 
   
  \item [\rm(2b)] When $(n_{j-1}+(r-1))+t_{j}<n< n_{j},$
 \begin{multline*}\mathcal D_{n}=\{  \nu=(\nu^{(j-1,r-1)}, \nu^{(j-1,r-1)}|_{t_{j}},  \nu_{(n_{j-1}+(r-1))+t_{j}+1},\dots,\nu_{n} ): \\ \nu^{(j-1,r-1)} \in \mathcal D_{n_{j-1}+(r-1)},  
\\    1\leq\nu_{u}\leq M,\ (n_{j-1}+(r-1))+t_{j}<u\leq n \} .
  \end{multline*}   
  
 \item [\rm(2c)] When $n= n_{j},$
  \begin{multline*} \mathcal D_{n}=  \{ \nu=(\nu^{(j-1,r-1)}, \nu^{(j-1,r-1)}|_{t_{j}},  \nu_{(n_{j-1}+(r-1))+t_{1}+1)},\dots, \nu_{n_{j}} ): \\ \gamma_{0}^{\tau({z_{j}})S_{l_{j}}h(z_{j})} \leq \nu_{n_{j}}<2\gamma_{0}^{\tau({z_{j}})S_{l_{j}}h(z_{j})}  \\ 
 {\rm and} \  1\leq\nu_{u}\leq M,\ {\rm for}\ (n_{j-1}+(r-1))+t_{j}<u<n_{j}  \}.  
  \end{multline*}
    where 
 $z_{j} \in I_{l_{j}} (\nu^{(j-1,r-1)}|_{t_{j}},  \nu_{(n_{j-1}+(r-1))+t_{j}+1} ,\cdots, \nu_{n_{j}-1}   ).$   
 
 \item [\rm(2d)] When $n=n_{j}+i$ where $1\leq i \leq r-2,$  
   \begin{multline*} \mathcal D_{n}=  \{ \nu=(\nu^{(j-1,r-1)}, \nu^{(j-1,r-1)}|_{t_{j}}, \nu_{(n_{j-1}+(r-1))+t_{j}+1)},\dots, \nu_{n_{j}+i} ): \\  \gamma_{i}^{\tau({z_{j}})S_{l_{j}}h(z_{j})} \leq \nu_{n_{j}+i}<2 \gamma_{i}^{\tau({z_{j}})S_{l_{j}}h(z_{j})}  {\rm where} \  1\leq i \leq r-2 , 
  \\   \gamma_{0}^{\tau({z_{j}})S_{l_{j}}h(z_{j})} \leq \nu_{n_{j}}<2 \gamma_{0}^{\tau({z_{j}})S_{l_{j}}h(z_{j})}    \\ {\rm and} \  1\leq\nu_{u}\leq M,\ {\rm for}\ (n_{j-1}+(r-1))+t_{j}<u< n_{j}  \}.  
  \end{multline*}  
  
  \item [\rm(2e)] When $n=n_{j}+(r-1),$  
  \begin{multline*} \mathcal D_{n}=  \{ \nu=(\nu^{(j-1,r-1)}, \nu^{(j-1,r-1)}|_{t_{j}},  \nu_{(n_{j-1}+(r-1))+t_{j}+1)},\dots, \nu_{n_{j}} ): \\   \left ( \frac{e}{\gamma_{0}\gamma_{1}\cdots\gamma_{r-2}}\right )          ^{\tau({z_{j}})S_{{l_{j}}}h(z_{j})}  \leq  \nu_{n_{j}+(r-1)}\\
  \qquad \qquad \qquad \qquad \qquad \qquad <   2\left ( \frac{e}{\gamma_{0}\gamma_{1}\cdots\gamma_{r-2}}\right )^{\tau({z_{j}})S_{l_{j}}h(z_{j})},  
 \\ \gamma_{i}^{\tau({z_{j}})S_{l_{j}}h(z_{j})} \leq \nu_{n_{j}+i}<2 \gamma_{i}^{\tau({z_{j}})S_{l_{j}}h(z_{j})} \ {\rm where} \  1\leq i \leq r-2 , 
  \\   \gamma_{0}^{\tau({z_{j}})S_{l_{j}}h(z_{j})} \leq \nu_{n_{j}}<2 \gamma_{0}^{\tau({z_{j}})S_{l_{j}}h(z_{j})} \\ {\rm and} \  1\leq\nu_{u}\leq M,\ {\rm for}\ (n_{j-1}+(r-1))+t_{j}<u< n_{j} \}.  
  \end{multline*}
  \end{itemize} 
 
 \noindent {\bf Fundamental cylinders:}
 For each $\nu=(\nu_{1}, \cdots,\nu_{n}) \in \mathcal D_{n},$ we define a refinement $J_{n}$ of $I_{n}$
 as the union of its sub-cylinders with nonempty intersection with $\mathcal E_{\infty}.$

\begin{itemize}
 \item [\rm(3a)] For $(n_{j-1}+(r-1))+t_{j}<n< n_{j}+1,$ define  
 \begin{equation}\label{eq3a}
 J_{n}(\nu)=\bigcup_{1\leq \nu_{n+1}\leq
M}                        I_{n+1}(\nu_{1},\dots ,\nu_{n},\nu_{n+1}). 
 \end{equation}
 \item [\rm(3b)] For $n=n_{j}-1,$ define
  \begin{equation}\label{eq3b}
 J_{n_{j}-1}(\nu)=\bigcup_{\gamma^{\tau (z_{j})S_{l_{j}}h(z_{j})}_{0}\leq \nu_{n_{j}}<
  2       \gamma^{\tau (z_{j})S_{l_{j}}h(z_{j})}_{0}                     }I_{n_{j}}(\nu_{1},\dots ,\nu_{n_{j}-1},\nu_{n_{j}}), 
 \end{equation}
  where 
 $z_{j} \in I_{l_{j}} (\nu^{(j-1,r-1)}|_{t_{j}},  \nu_{(n_{j-1}+(r-1))+t_{j}+1} ,\cdots, \nu_{n_{j}-1}   ).$   

  \item [\rm(3c)] For $n=n_{j}+i-1$ with $1\leq i \leq r-2,$ define
  \begin{equation}\label{eq3c}
 J_{n_{j}+i-1}(\nu)
 \bigcup_{\gamma^{\tau (z_{j})S_{l_{j}}h(z_{j})}_{i}\leq \nu_{n_{j}+i} \leq
  2       \gamma^{\tau (z_{j})S_{l_{j}}h(z_{j})}_{i}}I_{n_{j}+i} (\nu_{1},\dots, \nu_{n_{j+i}}). 
 \end{equation}
 \item [\rm(3d)] For $n=n_{j}+(r-2),$ define
 \begin{equation}\label{eq3d}
 \begin{aligned} &J_{n_{j}+(r-2)}(\nu)\\
 &=\bigcup_{ (\frac{e}{\gamma_{0}\cdots \gamma_{r-2}})^{\tau (z_{j})S_{l_{j}}h(z_{j})}\leq \nu_{n_{j}+(r-1)} \leq
  2    (\frac{e}{   \gamma_{0}\cdots \gamma_{r-2}})^{\tau (z_{j})S_{l_{j}}h(z_{j})}}I_{n_{j}+(r-1)}(w) ,
   \end{aligned}
   \end{equation}
 where $w=(\nu_{1},\dots, \nu_{n_{j}+(r-1)}).$
  \item [\rm(3e)]If $n_{j}+(r-1)\leq n \leq (n_{j}+(r-1))+t_{j+1},$ then by construction of $\mathcal E_{\infty},$ 
 \begin{equation}\label{eq3e}
 J_{n}(\nu)=I_{n_{j}+(r-1)+t_{j+1}}(\nu^{(j,r-1)},\nu^{(j,r-1)}|_{t_{j+1}}).
  \end{equation}
  \end{itemize} 

 Clearly,
 \begin{equation*}
\mathcal{E}_{\infty}=\bigcap_{n=1}^{\infty }\bigcup_{\nu \in D_{n}}J_{n}(\nu) .
\end{equation*}

 \subsubsection{Lengths of fundamental cylinders} \
In the following subsection we will estimate the lengths of the fundamental
cylinders defined above.

Let the continued fraction representation for any $x\in   \mathcal{E}_{\infty} $ be 
 $$[\nu^{(j-1,r-1)}, \nu^{(j-1,r-1)}|_{t_{j}},b^{(j)}_{1},\cdots,b^{(j)}_{m_{j}},a_{n_{j}},\cdots, a_{n_{j}+(r-1)},\cdots].$$ 
 
\noindent \textbf{\ I.} If $n=(n_{j}+(r-1))+t_{j+1},$ then by using  \eqref{eq P3}
\begin{multline*}
q_{(n_{j}+(r-1))+t_{j+1}}(\nu^{(j-1,r-1)}, \nu^{(j-1,r-1)}|_{t_{j}},\\
b^{(j)}_{1},\cdots,b^{(j)}_{m_{j}},a_{n_{j}},\cdots, a_{n_{j}+(r-1)},\nu^{(j,r-1)}|_{t_{j+1}})\\
\leq 2^{3r+2} q_{(n_{j-1}+(r-1))+t_{j}}(\nu^{(j-1,r-1)}, \nu^{(j-1,r-1)}|_{t_{j}})\cdot q_{m_{j}}(  b^{(j)}_{1},\cdots,b^{(j)}_{m_{j}} )\\ \cdot e^{\tau(z_{j})S_{l_{j}}h(z_{j})}\cdot q_{t_{j+1}}(\nu^{(j,r-1)}|_{t_{j+1}}).
\end{multline*}

Next from $\rm{(iii)}$ of Proposition \ref{pp3} and by the choice of $m_{j},$ 
\begin{align}
q_{(n_{j}+(r-1))+t_{j+1}}(x)
& \leq    q_{(n_{j-1}+(r-1))+t_{j}}(x)\cdot(q_{l_{j}}( z_{j} )e^{\tau(z_{j})S_{l_{j}}h(z_{j})})^{1+\epsilon} \notag \\
& \leq \prod^{j}_{k=1}(q_{l_{k}}(z_{k})e^{\tau(z_{k})    S_{l_{k}}  h(z_{k})})^{1+\epsilon}, \label{eqI}
\end{align}
where
$z_{k}\in I_{l_{k}}(  \nu^{(k-1,r-1)}|_{t_{k}},b^{(k)}_{1},\cdots,b^{(k)}_{m_{k}}) \text{ for all } 1\leq k \leq j.$ 

\noindent \textbf{\ II.} If $(n_{j}+(r-1))\leq n< (n_{j}+(r-1))+t_{j+1},$ then
{\allowdisplaybreaks
\begin{eqnarray*}
q_{n}(x)\leq           q_{(n_{j}+(r-1))+{t_{j+1}}}  (x)                \leq \prod^{j}_{k=1}(q_{l_{k}}(z_{k})\cdot e^{\tau(z_{k})    S_{l_{k}}  h(z_{k})})^{1+\epsilon}.
\end{eqnarray*}
}
\noindent \textbf{\ III.} When $(n_{j-1}+(r-1))+t_{j}\leq n\leq n_{j}-1,$ represent $n-(n_{j-1}+(r-1))-t_{j}$ by $l,$ then
\begin{align*}
q_{n}(x)&\leq      2     q_{(n_{j-1}+(r-1))+{t_{j}}} (x) \cdot q_{l}(b^{(j)}_{1},\cdots,b^{(j)}_{l})        
 \\     & \leq \prod^{j-1}_{k=1}(q_{l_{k}}(z_{k})e^{\tau(z_{k})    S_{l_{k}}  h(z_{k})})^{1+\epsilon}\cdot q_{l}(b^{(j)}_{1},\cdots,b^{(j)}_{l}).
\end{align*}

Now we calculate the lengths of fundamental cylinders for different cases \eqref{eq3a}--\eqref{eq3e} defined above.
\\

 \noindent \textbf{\ I.} If $(n_{j-1}+(r-1))\leq n\leq (n_{j-1}+(r-1))+t_{j},$ then from  \eqref{cyle}, \eqref{eq3e} and \eqref{eqI}, 
\begin{eqnarray*}
\left\vert J_{n}(x)\right\vert &=&\left\vert I_{(n_{j-1}+(r-1))+t_{j}}(x)
\right\vert \geq \frac{1}{2q^{2}_ {(n_{j-1}+(r-1))+t_{j}}(x) }
 \notag \\
&\geq&   \frac{1}{2}   \prod^{j-1}_{k=1}(q_{l_{k}}(z_{k})\cdot e^{\tau(z_{k})    S_{l_{k}}  h(z_{k})})^{-2(1+\epsilon)}. 
\end{eqnarray*}
 \\
  \noindent \textbf{\ II.} If $(n_{j-1}+(r-1))+t_{j}< n< n_{j}-1$ and $l=n-(n_{j-1}+(r-1))-t_{j}-1,$ then from \eqref{cyle} and \eqref{eq3a},
  \begin{align*}
|J_{n}(x)|\geq \frac{1}{6q^{2}_{n} (x)}
  \geq \frac{1}{6}{\prod^{j}_{k=1}(q_{l_{k}}(z_{k})\cdot e^{\tau(z_{k})    S_{l_{k}}  h(z_{k})})^{-2(1+\epsilon)}\cdot q^{-2}_{l}(b^{(j)}_{1},\cdots,b^{(j)}_{l})}.
\end{align*}
 \\
 \noindent \textbf{III.} If $n=n_{j}-1,$ then following the similar steps
 as for \noindent \textbf{I} and using \eqref{eq3b} 
\begin{align*}
|J_{n_{j}-1}(x)|&\geq \frac{1}{6\nu_{n_{j}}(x)q^{2}_{n_{j}-1} (x)}\geq \frac{1}{6\gamma^{\tau(z_{j})S_{l_{j}}h(z_{j})}_{0}q^{2}_{n_{j}-1} (x)}
\\&  \geq \frac{1}{24  \gamma^{\tau(z_{j})S_{l_{j}}h(z_{j})} _{0}q^{2}_{l_{j}}(x)}\cdot \prod^{j-1}_{k=1}(q_{l_{k}}(z_{k})\cdot e^{\tau(z_{k})    S_{l_{k}}  h(z_{k})})^{-2(1+\epsilon)}.\end{align*}
\\
 \noindent \textbf{IV.} If $n=n_{j}+i-1 \ {\rm where} \ 1\leq i \leq r-2$ then from \eqref{eq3c} and following the similar steps
 as for \noindent \textbf{I,} 
{\allowdisplaybreaks
\begin{multline*}
|J_{n_{j}+i-1}(x)|\geq \frac{1}{6\nu_{n_{j}+i}(x)q^{2}_{n_{j}+i-1} (x)}\geq \frac{1}{6\gamma^{\tau(z_{j})S_{l_{j}}h(z_{j})}_{i}q^{2}_{n_{j}+i-1} (x)}
\\ \geq \frac{1}{6\cdot4^{i}   \gamma^{\tau(z_{j})S_{l_{j}}h(z_{j})}_{i}      (\gamma_{0}\cdots \gamma_{i-1})^{2\tau(z_{j})S_{l_{j}}h{(z_{j})}}q^{2}_{n_{j}-1}(x)}
\\ \geq \frac{1}{6\cdot4^{i}  \gamma^{\tau(z_{j})S_{l_{j}}h(z_{j})}_{i}      (\gamma_{0}\cdots \gamma_{i-1})^{2\tau(z_{j})S_{l_{j}}h{(z_{j})}}q^{2}_{l_{j}}(z_{j})}\\
\qquad \qquad \cdot \prod^{j-1}_{k=1}(q_{l_{k}}(z_{k})e^{\tau(z_{k})    S_{l_{k}}  h(z_{k})})^{-2(1+\epsilon)}.
\end{multline*}
 }
 \\
  \noindent \textbf{V.} If $n=n_{j}+(r-2),$ then from \eqref {eq3d}
\begin{multline*}
|J_{n_{j}+(r-2)}(x)|\geq \frac{1}{6\nu_{n_{j}+(r-1)}(x)q^{2}_{n_{j}+(r-2)}(x)} 
\\ \geq \frac{1}{6\cdot4^{r-1}(e\gamma_{0}\cdots \gamma_{r-2})^{\tau(z_{j})S_{l_{j}}h{(z_{j})}}q^{2}_{n_{j}-1}(x)}
\\  \geq \frac{1}{6\cdot4^{r}(e\gamma_{0}\cdots \gamma_{r-2})^{\tau(z_{j})S_{l_{j}}h{(z_{j})}}q^{2}_{l_{j}}(z_{j})} \\ \cdot \prod^{j-1}_{k=1}(q_{l_{k}}(z_{k})e^{\tau(z_{k})    S_{l_{k}}  h(z_{k})})^{-2(1+\epsilon)}.\end{multline*}
 
 \subsubsection{Supporting measure} \ 

We will define a probability measure supported on the set $\mathcal E_{\infty}.$

Define $s_{j}:=s^{(r)}_{(t_{j},m_{j}),M}$ to be the solution of
 \begin{equation*}
\sum_{\substack{ a_{1}=\nu^{(j-1,r-1)}_{1} ,\cdots, a_{t_{j}} =\nu^{(j-1,r-1)}_{t_{j}} \\
 1\leq b^{(j)}_{1},\cdots, b^{(j)}_{m_j} \leq M}}\frac{1}{e^{g_{r}(s)\tau{(z_{j})}S_{l_{j}}h(z_{j})}q^{2s}_{l_{j}}(z_{j})}=1\end{equation*}
where $z_{j}\in I_{l_{j}} (\nu^{(j-1,r-1)}|_{t_{j}}, b^{(j)}_{1},\cdots,b^{(j)}_{m_{j}})$ and 
the sequences $\{t_{j}\}_{j\geq1}$ and $\{m_{j}\}_{j\geq1}$ are defined previously. 

Consequently from \eqref{eqpro},
 \begin{equation}\label{meacond}
 \sum_{\substack{ a_{1}=\nu^{(j-1,r-1)}_{1} ,\cdots, a_{t_{j}} =\nu^{(j-1,r-1)}_{t_{j}}\\1\leq b^{(j)}_{1},\cdots, b^{(j)}_{m_j}\leq M }}\left(\frac{1}{\gamma_{0}^{\tau{(z_{j})}S_{l_{j}}h(z_{j})}q^{2}_{l_{j}}(z_{j})} \right)^{s}=1,
 \end{equation}
where $z_{j}\in I_{l_{j}} (\nu^{(j-1,r-1)}|_{t_{j}}, b^{(j)}_{1},\cdots,b^{(j)}_{m_{j}}).    $

Equality \eqref{meacond} induces a measure $\mu$ on basic cylinder of order $t_{j}+m_{j}$ if we consider 
\begin{equation*}
\mu(I_{n_{j}+t_{j}}(        a_{1} ,\cdots, a_{t_{j}}, b^{(j)}_{1},\cdots, b^{(j)}_{m_j}  ))= \left(\frac{1}{\gamma_{0}^{\tau{(z_{j})}S_{l_{j}}h(z_{j})}q^{2}_{l_{j}}(z_{j})} \right)^{s_{j}},
 \end{equation*} for each
$ a_{1}=\nu^{(j-1,r-1)}_{1} ,\cdots, a_{t_{j}} =\nu^{(j-1,r-1)}_{t_{j}},1\leq b^{(j)}_{1},\cdots, b^{(j)}_{m_j}\leq M. $

We will start by assuming that the measure of $   I_{{n_{j-1}+(r-1)}}(x)  \in  \mathcal E_{\infty}$ has been defined as
 \begin{equation*} 
 \mu \left( I_{{n_{j-1}+(r-1)}
}(x)\right)  =\prod^{j-1}_{k=1}        \left(               \left( \frac{1}{\gamma^{\tau(z_{k})S_{l_{k}}h(z_{k})} _{0}q_{l_{k}}^{2}(z_{k})}\right)
 ^{s_{k}}            \frac{1}{e^{\tau(z_{k})S_{l_{k}}h(z_{k})}  }                                    \right),\end{equation*}
where $z_{k}\in I_{l_{k}} (\nu^{(k-1,r-1)}|_{t_{k}}, b^{(k)}_{1},\cdots,b^{(k)}_{m_{k}})$ for all $1\leq k \leq j-1.$
\\

 \noindent \textbf{Case 1:} $n_{j-1}+(r-1)< n \leq n_{j-1}+(r-1)+t_{j}.$ As the basic cylinder of order ${n_{j-1}+(r-1)}$ contains only one sub-cylinder of order $n$ with a nonempty intersection with $\mathcal E_{\infty},$ therefore  
 \begin{align*}
  \mu (I_{n}(x))=\mu \left( I_{n_{j-1}+(r-1)}(x)\right).
  \end{align*}

 \noindent \textbf{Case 2:} $n=n_{j}-1.$ Let 
  \begin{equation*}
\mu (I_{n_{j}-1}(x))=\mu \left( I_{{n_{j-1}+(r-1)}
}(x)\right) \cdot 
\left( \frac{1}{\gamma^{\tau(z_{j})S_{l_{j}}h(z_{j})} _{0}q_{l_{j}}^{2}(z_{j})}\right)
 ^{s_{j}}.\end{equation*}  
    Next we uniformly distribute the measure of $I_{n_{j}-1}(x)$ on its sub-cylinders.
    
  \noindent \textbf{Case 3:} $n=n_{j}+i-1,$ where $1\leq i\leq r-1.$    
   \begin{align*}
  \mu (I_{n_{j}+i-1}(x))&=\mu \left( I_{n_{j}+i-2}(x)\right) \cdot         \frac{1}{\gamma_{i-1}^{\tau(z_{j})S_{l_{j}}h(z_{j})}}  \\&
  =\mu \left( I_{n_{j}-1}(x)\right) \cdot         \frac{1}{(\gamma_{0}\cdots\gamma_{i-1})^{\tau(z_{j})S_{l_{j}}h(z_{j})}}.
  \end{align*}
  
  \noindent \textbf{Case 4:}  $n=n_{j}+(r-1).$   
 \begin{align*}
  \mu (I_{n_{j}+r-1}(x))&= (\frac{\gamma_{0}\cdots\gamma_{r-2}}{e})^{\tau(z_{j})S_{l_{j}}h(z_{j})}\mu \left( I_{n_{j}+(r-2)}(x)\right)  
  \\&= (\frac{\gamma_{0}\cdots\gamma_{r-2}}{e})^{\tau(z_{j})S_{l_{j}}h(z_{j})}   {\frac{1}{  ( \gamma_{0}\cdots\gamma_{r-2})     ^{\tau(z_{j})S_{l_{j}}h(z_{j})}}}\\
 &{} \qquad \qquad \qquad \qquad \cdot \mu \left( I_{n_{j}+(r-2)}(x)\right)    \\  &=\frac{1}{e^ {\tau(z_{j})S_{l_{j}}h(z_{j}) }} \mu \left( I_{n_{j}-1}(x)\right).
  \end{align*}   
  The measure of other basic cylinders of order less than $ n_{j}-1$ is followed by the consistency property that a measure should satisfy. 
  
  \noindent For any $n_{j-1}+(r-1)+t_{j}< n \leq n_{j}-1,$ let  \begin{align*}
  \mu (I_{n}(x))&=\sum_{I_{n_{j}-1}(x)\subset I_{n}(x)} \mu \left( I_{n_{j}-1}(x)\right) .  \end{align*}      
  
 \subsubsection{The H\"{o}lder exponent of the measure $\protect\mu $} \
 
In this part we will compare the measure of fundamental cylinders with their lengths.

\noindent \textbf{Case 1:} $n=n_{j}-1.$ 
 \begin{multline} 
 \mu \left( J_{{n_{j}-1}
}(x)\right) =\prod^{j-1}_{k=1}        \left(               \left( \frac{1}{\gamma^{\tau(z_{k})S_{l_{k}}h(z_{k})} _{0}q_{l_{k}}^{2}(z_{k})}\right)
 ^{s_{k}}            \frac{1}{e^{\tau(z_{k})S_{l_{k}}h(z_{k})}  }   \right)\\ \cdot \left( \frac{1}{\gamma^{\tau(z_{j})S_{l_{j}}h(z_{j})} _{0}q_{l_{j}}^{2}(z_{j})}\right)
 ^{s_j}\notag
  \\ \leq
  \prod^{j-1}_{k=1}         \left( \frac{1}{\gamma^{s_{k}\tau(z_{k})S_{l_{k}}h(z_{k})} _{0} e^{\tau(z_{k})S_{l_{k}}h(z_{k})} q_{l_{k}}^{2s_k}(z_{k})}\right)
   \cdot \left( \frac{1}{\gamma^{\tau(z_{j})S_{l_{j}}h(z_{j})} _{0}q_{l_{j}}^{2}(z_{j})}\right)
 ^{s^{(r)}_M-3\epsilon}  \notag
 \\  \leq \label{zz}
 \prod^{j-1}_{k=1}         \left( \frac{1}{ e^{2s_{k}\tau(z_{k})S_{l_{k}}h(z_{k})} q_{l_{k}}^{2s_k}(z_{k})}\right)
   \cdot \left( \frac{1}{\gamma^{\tau(z_{j})S_{l_{j}}h(z_{j})} _{0}q_{l_{j}}^{2}(z_{j})}\right)
 ^{s^{(r)}_M-3\epsilon}  
 \\ \leq
\left( \prod^{j-1}_{k=1}         \left( \frac{1}{ e^{2\tau(z_{k})S_{l_{k}}h(z_{k})} q_{l_{k}}^{2}(z_{k})}\right)^{1+\epsilon} \right)^{\frac{s^{(r)}_M-3\epsilon}{1+\epsilon}}
   \cdot \left( \frac{1}{\gamma^{\tau(z_{j})S_{l_{j}}h(z_{j})} _{0}q_{l_{j}}^{2}(z_{j})}\right)
 ^\frac{s^{(r)}_{M}-3\epsilon}{1+\epsilon} \notag
  \\ \leq 24 |J_{n_{j}-1}(x)| ^     \frac{s^{(r)}_{M}-3\epsilon}{1+\epsilon}. \notag \end{multline}

From Corollary \ref{limcor}, we have $| s_{j}-s^{(r)}_{M}| \leq 3\epsilon$ which further implies that $s^{(r)}_{M} -3\epsilon \leq s_{j}.$ In \eqref{zz}  we have used the fact that since  $1\leq \gamma_{0}\cdots\gamma_{r-2}\leq e$ and $\frac{e}{ \gamma_{0}\cdots\gamma_{r-2}}\geq \left(\frac{e}{ \gamma_{0}\cdots\gamma_{r-2}  }  \right)^{s}$
for any $0<s<1,$ we have $e\gamma^{s}_{0}\geq e^{2s}.$ Therefore it is also true for $s_{k}$.

 \noindent \textbf{Case 2:} $n=n_{j}+i-1,$ where $1\leq i\leq r-2.$    
{\allowdisplaybreaks  
  \begin{align}
  \mu (J_{n_{j}+i-1}(x))
   &{}=   \mu \left( I_{n_{j}-1}(x)\right) \cdot         \frac{1}{(\gamma_{0}\cdots\gamma_{i-1})^{\tau(z_{j})S_{l_{j}}h(z_{j})}} \notag
 \\ &{} \leq 
\left( \prod^{j-1}_{k=1}         \left( \frac{1}{ e^{2\tau(z_{k})S_{l_{k}}h(z_{k})} q_{l_{k}}^{2}(z_{k})}\right)^{1+\epsilon} \right)^{\frac{s^{(r)}_M-3\epsilon}{1+\epsilon}} \notag \\ & \cdot \left( \frac{1}{\gamma^{\tau(z_{j})S_{l_{j}}h(z_{j})} _{0}q_{l_{j}}^{2}(z_{j})}\right)
 ^\frac{s^{(r)}_{M}-3\epsilon}{1+\epsilon}  \cdot         \frac{1}{(\gamma_{0}\cdots\gamma_{i-1})^{\tau(z_{j})S_{l_{j}}h(z_{j})}}                    \notag  
 \\  &{} \leq  
\left( \prod^{j-1}_{k=1}         \left( \frac{1}{ e^{2\tau(z_{k})S_{l_{k}}h(z_{k})} q_{l_{k}}^{2}(z_{k})}\right)^{1+\epsilon} \right)^{\frac{s^{(r)}_M-3\epsilon}{1+\epsilon}}  \\ &{}  \cdot \left( \frac{1}{\gamma^{\tau(z_{j})S_{l_{j}}h(z_{j})} _{0}q_{l_{j}}^{2}(z_{j})}\right) 
 ^\frac{s^{(r)}_{M}-3\epsilon}{1+\epsilon} \frac{1}{(\gamma_{0}(\gamma_{1}\cdots\gamma_{i-1})^{2}\gamma_{i})^{\tau(z_{j})S_{l_{j}}h(z_{j})}}  \label{HXC2iv}
  \\ & \leq  
\left( \prod^{j-1}_{k=1}         \left( \frac{1}{ e^{2\tau(z_{k})S_{l_{k}}h(z_{k})} q_{l_{k}}^{2}(z_{k})}\right)^{1+\epsilon} \right)^{\frac{s^{(r)}_M-3\epsilon}{1+\epsilon}} \notag
\\ &{} \cdot \left( \frac{1}{(\gamma_{0}\gamma_{1} \cdots\gamma_{i-1})^{2\tau(z_{j})S_{l_{j}}h(z_{j})}  \gamma_{i}^{\tau(z_{j})S_{l_{j}}h(z_{j}) }          q_{l_{j}}^{2}(z_{j})}\right) ^{\frac{s^{(r)}_M-3\epsilon}{1+\epsilon}} \notag
 \\ &{} \leq   6.4^{i+1} |J_{n_{j}+i-1}(x)| ^     \frac{s^{(r)}_{M}-3\epsilon}{1+\epsilon}  \notag\\ & \leq 6.4^{r-3} |J_{n_{j}+i-1}(x)| ^     \frac{s^{(r)}_{M}-3\epsilon}{1+\epsilon}. \notag
    \end{align}
    } 
  Equation \eqref{HXC2iv} is obtained 
 by using the fact that $\frac{1}{\gamma_{0}\gamma_{1}\cdots\gamma_{i-1}}\leq (\frac{1}{\gamma_{0}(\gamma_{1}\cdots\gamma_{i-1})^{2}\gamma_{i}})^{s}$ for any $0<s<1.$ 
 
 \noindent \textbf{Case 3:} $n=n_{j}+r-2.$   
 {\allowdisplaybreaks
 \begin{align}
&  \mu (J_{n_{j}+r-2}(x))=\mu \left( J_{n_{j}-1}(x)\right) \cdot         \frac{1}{(\gamma_{0}\cdots\gamma_{r-2})^{\tau(z_{j})S_{l_{j}}h(z_{j})}} \notag
 \\&  \leq
\left( \prod^{j-1}_{k=1}         \left( \frac{1}{ e^{2\tau(z_{k})S_{l_{k}}h(z_{k})} q_{l_{k}}^{2}(z_{k})}\right)^{1+\epsilon} \right)^{\frac{s^{(r)}_M-3\epsilon}{1+\epsilon}} \notag
    \cdot \left( \frac{1}{\gamma^{\tau(z_{j})S_{l_{j}}h(z_{j})} _{0}q_{l_{j}}^{2}(z_{j})}\right)
 ^\frac{s^{(r)}_{M}-3\epsilon}{1+\epsilon} \\ & \quad \quad  \cdot       \left(  \frac{1}{(e\gamma_{1}\cdots\gamma_{r-2})^{\tau(z_{j})S_{l_{j}}h(z_{j})}}  \right)  ^{s_{j}}                \notag  
 \\ & \leq 
\left( \prod^{j-1}_{k=1}         \left( \frac{1}{ e^{2\tau(z_{k})S_{l_{k}}h(z_{k})} q_{l_{k}}^{2}(z_{k})}\right)^{1+\epsilon} \right)^{\frac{s^{(r)}_M-3\epsilon}{1+\epsilon}} \notag
 \\ &  \quad \quad   \cdot \left( \frac{1}{  (e\gamma_{0}\cdots\gamma_{r-2}) ^{\tau(z_{j})S_{l_{j}}h(z_{j})} q_{l_{j}}^{2}(z_{j})}\right)
 ^\frac{s^{(r)}_{M}-3\epsilon}{1+\epsilon}  \notag  
   \\ & \leq 6.4^{r} |J_{n_{j}+r-2}(x)| ^     \frac{s^{(r)}_{M}-3\epsilon}{1+\epsilon}. \notag
    \end{align} 
    }
 
  \noindent \textbf{Case 4:} $n_{j}+(r-1)\leq n \leq n_{j}+t_{j+1}.$   
 {\allowdisplaybreaks 
   \begin{align*}
  \mu (J_{n}(x))&     =\prod^{j-1}_{k=1}        \left(               \left( \frac{1}{\gamma^{\tau(z_{k})S_{l_{k}}h(z_{k})} _{0}q_{l_{k}}^{2}(z_{k})}\right)
 ^{s_{k}}            \frac{1}{e^{\tau(z_{k})S_{l_{k}}h(z_{k})}  }   \right) \\ &\cdot \left( \frac{1}{\gamma^{\tau(z_{j})S_{l_{j}}h(z_{j})} _{0}q_{l_{j}}^{2}(z_{j})}\right)
 ^{s_j}             \frac{1}{e^ {\tau(z_{j})S_{l_{j}}h(z_{j}) }} 
 \\& = \prod^{j}_{k=1}        \left(               \left( \frac{1}{\gamma^{\tau(z_{k})S_{l_{k}}h(z_{k})} _{0}q_{l_{k}}^{2}(z_{k})}\right)
 ^{s_{k}}            \frac{1}{e^{\tau(z_{k})S_{l_{k}}h(z_{k})}  }   \right)  
 \\ & \leq  \prod^{j}_{k=1}                   \left( \frac{1}{e^{2\tau(z_{k})S_{l_{k}}h(z_{k})} _{0}q_{l_{k}}^{2}(z_{k})}\right)
 ^{s_{k}}     
 \\ & \leq  \prod^{j}_{k=1}                   \left( \frac{1}{e^{2\tau(z_{k})S_{l_{k}}h(z_{k})} _{0}q_{l_{k}}^{2}(z_{k})}\right)
 ^ {s^{(r)}_{M}-3\epsilon}  
 \\ & \leq \left( \prod^{j}_{k=1}                   \left( \frac{1}{e^{2\tau(z_{k})S_{l_{k}}h(z_{k})} _{0}q_{l_{k}}^{2}(z_{k})}\right)^{1+\epsilon}
 \right)^\frac {s^{(r)}_{M}-3\epsilon}  {1+\epsilon} \leq 2 |J_{n}(x)| ^     \frac{s^{(r)}_{M}-3\epsilon}{1+\epsilon}. \end{align*} 
  }
\subsubsection{\noindent {Gap estimation.}} \
  
  Let $x\in   \mathcal{E}_{\infty}.$ In this section we estimate the gap between $J_{n}(x)$ and 
its adjacent fundamental cylinder $J_{n}(x')$ of the same order $n$. Assume that $a_{i}(x)=a_{i}(x')$ for all $1\leq i < n.$ These gaps are
helpful for estimating the measure on general balls. Also as $J_{n}(x)$ and $J_{n}(x')$ are adjacent, we have $\vert a_{n}(x) - a_{n}(x')\vert=1.$

Let the left and the right gap between $J_{n}(x)$ and
its adjacent fundamental cylinder at each side be represented by $g_{n}^{L
}(x)$ and $g_{n}^{R}(x)$ respectively. 

Denote by $g^{L,R}_{n}(x)$ the minimum distance between $
J_{n}(x)$ and its adjacent cylinder of the same order $n$,
that is,
\begin{equation*}
{g^{L,R}_{n}}(x)=\min \{g_{n}^{L }(x),g_{n}^{R}(x)\}.
\end{equation*}
 Without loss of generality we assume that $n$ is
even and estimate $g_{n}^{R}(x)$ only, since if $n$ is odd then for $g_{n}^{L
}(x)$ we can carry out the estimation in almost the same
way.   
  
 \noindent \textbf{Gap I.} When $(n_{j-1}+(r-1))+t_{j}< n< n_{j}-1$ for
all $j\geq 1,$ 
\begin{eqnarray*}
g^{R}_{n}\left( x\right) &\geq& 
\sum _{a_{n+1}>M}\vert I_{n+1}\left( a_{1},a_{2},\ldots ,a_{n-1},a_{n}+1,a_{n+1}\right) \vert \\ 
&=&\frac{\left( M+1\right)
\left( p_{n}+p_{n-1}\right) +p_{n-1}}{\left( M+1\right) \left(
q_{n}+q_{n-1}\right) +q_{n-1}}-\frac{p_{n}+p_{n-1}}{q_{n}+q_{n-1}} \\
&=&\frac{1}{\left( \left( M+1\right) \left( q_{n}+q_{n-1}\right)
+q_{n-1}\right) \left( q_{n}+q_{n-1}\right) }
\\ & \geq &\frac{1}{3Mq^{2}_{n}}\geq   \frac{1}{3M}\vert I_{n}(x)\vert .
\end{eqnarray*}
  \noindent \textbf{Gap II.} When $n=n_{j}+i-1 \ {\rm where} \ 0\leq i \leq r-2,$  
  \begin{eqnarray*}
g^{R}_{n}\left(x\right) 
 &\geq & \frac {p_{n}+p_{n-1}}{q_{n}+q_{n-1}}-\frac{ \gamma^{\tau (z_{j})S_{l_{j}}h(z_{j})}_{i}p_{n}+p_{n-1}}{
   \gamma^{\tau (z_{j})S_{l_{j}}h(z_{j})}_{i}q_{n} +q_{n-1}}
 \\
&=&\frac{     \gamma^{\tau (z_{j})S_{l_{j}}h(z_{j})}_{i}   -1}{\left(    \gamma^{\tau (z_{j})S_{l_{j}}h(z_{j})}_{i}       q_{n}+q_{n-1}\right) \left(
q_{n}+q_{n-1}\right) } \geq \frac{     {     \gamma^{\tau (z_{j})S_{l_{j}}h(z_{j})}_{i}   -1     }     }{4   \gamma^{\tau (z_{j})S_{l_{j}}h(z_{j})}_{i} q^2_{n}   } 
\\  &\geq &\frac{     {     \gamma^{\tau (z_{j})S_{l_{j}}h(z_{j})}_{i}       }     }{8  \gamma^{\tau (z_{j})S_{l_{j}}h(z_{j})}_{i} q^2_{n}   } \geq \frac{1}{8} \vert I_{n}(x)\vert.
\end{eqnarray*}
   \noindent \textbf{Gap III.} When $n=n_{j}+r-2,$ 
  \begin{eqnarray*}
g^{R}_{n}\left(x\right) 
&\geq&\frac{     (\frac{e}{\gamma_{0}\cdots \gamma_{r-2}})^{\tau (z_{j})S_{l_{j}}h(z_{j})}   -1}{\left(      (\frac{e}{\gamma_{0}\cdots \gamma_{r-2}})^{\tau (z_{j})S_{l_{j}}h(z_{j})}            q_{n}+q_{n-1}\right) \left(
q_{n}+q_{n-1}\right) } 
\\  &\geq &\frac{         1    }{8   q^2_{n}   } \geq \frac{1}{8} \vert I_{n}(x)\vert.
\end{eqnarray*}
 \noindent \textbf{Gap IV.} If $(n_{j}+(r-1))\leq n\leq (n_{j}+(r-1))+t_{j+1}$ then note that $J_{n}(x)$ is a small part of $I_{n_{j}+(r-1)}(x)$ as $I_{(n_{j}+(r-1))+t_{j+1}}(x) \subset I_   {(n_{j}+(r-1))+2}(x).$  Therefore, the right gap is larger then the distance between the right endpoints of $J_{n}(x)$ and that of $I_{n_{j}+(r-1)}(x).$ 
\begin{eqnarray*}
g^{R}_{n}\left(x\right) 
&\geq&\left\vert I_{(n_{j}+(r-1))+2}(x)
\right\vert \geq \frac{1}{2q^{2}_ {(n_{j}+(r-1))+2}(x) }\geq \frac{1}{32 a^{2}_{1}a^{2}_{2} q^{2}_{n_{j}+(r-1)}(x)}
 \notag \\
&\geq&   \frac{1}{32 a^{2}_{1}a^{2}_{2} q^{2}_{(n_{j}+(r-1))+t_{j}}(x)} \geq  \frac{1}{32 a^{2}_{1}a^{2}_{2}} \vert  I_{(n_{j}+(r-1))+t_{j}} (x) \vert 
\\ &=&  \frac{1}{32 a^{2}_{1}a^{2}_{2}} \vert  J_{(n_{j}+(r-1))+t_{j}}(x)  \vert   ,    \end{eqnarray*}
   where $a_{1}$ represents $a_{(n_{j}+(r-1))+1}(x)$ and $a_{2}$ represents $a_{(n_{j}+(r-1))+2}  (x).$

  \subsubsection{The measure $\protect\mu $ on general ball $B(x,d)$}\

We now estimate the measure $\mu $ on any ball $
B(x,d)$ with radius $d$ and centred at~$x.$ Fix $x\in \mathcal{E}_{\infty}.$ There exists a unique sequence $(\nu_{1},\nu_{2},\cdots
\nu_{n},\cdots)$ such that for each $n\geq1,$
 {$x \in J_{n}(\nu_{1},\cdots ,\nu_{n})$ where $(\nu_{1}\cdots,\nu_{n})\in     \mathcal D_{n}     $ and 
$g^{R}_{n+1}(x)\leq d<g^{R}_{n}(x).$
Clearly, $B(x,d)$ can intersect only
one fundamental cylinder of order $n,$ i.e., $J_{n}(\nu_{1},\ldots ,\nu_{n}).$

\noindent \textbf{Case I:}  $(n_{j-1}+(r-1))+t_{j}< n< n_{j}-1$, for
all $j\geq 1$ or $ n=n_{j}+t_{j+1}.$
Since in this case $1\leq a_{n}(x)\leq M$ and $|J_{n}(x)|\leq 1/q_{n}^{2}$ thus  
\begin{align*}
\mu (B(x,d))& \leq \mu (J_{n}(x))\leq c|J_{n}(x)|^{          \frac{s^{(r)}_{M}-3\epsilon}{1+\epsilon}             }\\
& \leq c\left( \frac{1}{q_{n}^{2}}\right) ^      {          \frac{s^{(r)}_{M}-3\epsilon}{1+\epsilon}             }\              \\
& \leq c4M^{2}\left( \frac{1}{q_{n+1}^{2}}\right) ^{          \frac{s^{(r)}_{M}-3\epsilon}{1+\epsilon}             }\\
& \leq c8M^{2}  {\vert I_{n+1} {(x) }\vert^     \frac{s^{(r)}_{M}-3\epsilon}{1+\epsilon}             }   \\
& \leq c48M^{3}\big({g^{R}_{n+1}}(x)\big)^    {          \frac{s^{(r)}_{M}-3\epsilon}{1+\epsilon}             }      \\
& \leq cc_{0}^{3}d^    {          \frac{s^{(r)}_{M}-3\epsilon}{1+\epsilon}             }        .
\end{align*}

\noindent \textbf{Case II:}  $n=n_{j}+i-1 \ {\rm where} \ 0\leq i \leq r-2.$ 
Since 
\begin{equation*}
\frac{1}{8\gamma^{2\tau(z_{j})S_{l_{j}}h(z_{j})}_{i}q^{2}_{n_{j}+{i-1}}(x)}\leq|I_{n_{j}+i}(x)|\leq 8 g^{R}_{n_{j}+i}(x)\le 8d\end{equation*}  implies 
\begin{equation*}
 1\leq 64d \gamma^{2\tau(z_{j})S_{l_{j}}h(z_{j})}_{i}q^{2}_{n_{j}+{i-1}}(x),\end{equation*}
 the number of fundamental cylinders of order $n_{j}+i$ contained in $
J_{n_{j}+i-1}(x)$ that the ball $B(x,d)$ intersects is at
most 
\begin{align*}
\frac{2d}{\vert I_{n_{j}+{i}}(x)\vert}+2&\leq 16d\gamma^{2\tau(z_{j})S_{l_{j}}h(z_{j})}_{i}q^{2}_{n_{j}+{i-1}}(x) +2^{7} d    \gamma^{2\tau(z_{j})S_{l_{j}}h(z_{j})}_{i}\\&=c_{0}d\gamma^{2\tau(z_{j})S_{l_{j}}h(z_{j})}_{i}q^{2}_{n_{j}+{i-1}}(x). \end{align*}

Therefore, 
\begin{align*}
&\mu (B(x,d))\\& \leq \min \Big\{\mu (J_{n_{j}+i-1}(x)),c_{0}d\gamma^{2\tau(z_{j})S_{l_{j}}h(z_{j})}_{i}q^{2}_{n_{j}+{i-1}}(x) \mu (J_{n_{j}+i}(x))\Big\} \\
& \leq \mu (J_{n_{j}+{i-1}}(x))\min \Big\{1,            c_{0}d\gamma^{2\tau(z_{j})S_{l_{j}}h(z_{j})}_{i}q^{2}_{n_{j}+{i-1}}(x) \frac{1}{   \gamma^{\tau(z_{j})S_{l_{j}}h(z_{j})}_{i}          }                              \Big\} \\
& \leq 6\cdot 4^{i+1}|J_{n_{j}+{i-1}}(x)|^    \frac{s^{(r)}_{M}-3\epsilon}{1+\epsilon}\min \Big\{1,                                   c_{0}d\gamma^{\tau(z_{j})S_{l_{j}}h(z_{j})}_{i}q^{2}_{n_{j}+{i-1}} (x)     \Big\} \\
& \leq c\Big(\frac{1}{   \gamma^{\tau(z_{j})S_{l_{j}}h(z_{j})}_{i}q^{2}_{n_{j}+{i-1}} (x)                      }\Big)^   \frac{s^{(r)}_{M}-3\epsilon}{1+\epsilon}   \Big( c_{0}d\gamma^{\tau(z_{j})S_{l_{j}}h(z_{j})}_{i}q^{2}_{n_{j}+{i-1}}(x)  \Big) ^  \frac{s^{(r)}_{M}-3\epsilon}{1+\epsilon}              \\
& \leq cc_{0}d^{          \frac{s^{(r)}_{M}-3\epsilon}{1+\epsilon}             }.
\end{align*}
Here we have used the fact that $\min \{a,b\}\leq a^{1-s}b^{s}$ for any $a,b>0$ and $0\leq s\leq
1$.

  \noindent \textbf{Case III:}  $n=n_{j}+r-2.$ 
As
\begin{equation*}
\frac{(\gamma_{0}\cdots \gamma_{r-2})^{{2\tau(z_{j})S_{l_{j}}h(z_{j})}}}{8e^{2\tau(z_{j})S_{l_{j}}h(z_{j})}q^{2}_{n_{j}+r-2}(x)}\leq|I_{n_{j}+r-1}(x)|\leq 8 g^{R}_{n_{j}+r-1}(x)\le 8d, \end{equation*}
 the number of fundamental cylinders of order $n_{j}+r-1$ contained in cylinder $
J_{n_{j}+r-2}(x)$ that the ball $B(x,d)$ intersects is at
most 
\begin{equation*}
\frac{2d}{\vert I_{n_{j}+r-1}(x)\vert}+2\leq c_{0}d      \frac{e^{{2\tau(z_{j})S_{l_{j}}h(z_{j})}}}{   (\gamma_{0}\cdots \gamma_{r-2})  ^{2\tau(z_{j})S_{l_{j}}h(z_{j})} }      q^{2}_{n_{j}+{r-1}}(x).               \end{equation*}

Therefore, 
{\allowdisplaybreaks
\begin{align*}
&\mu (B(x,d))\\ \leq & \min \begin{aligned} \Big\{ \mu (J_{n_{j}+r-2}(x)),c_{0}d       \frac{e^{{2\tau(z_{j})S_{l_{j}}h(z_{j})}}}{   (\gamma_{0}\cdots \gamma_{r-2})  ^{2\tau(z_{j})S_{l_{j}}h(z_{j})} } &   q^{2}_{n_{j}+{r-2}} (x)    \\&       \mu (J_{n_{j}+r-1}(x)) \Big\} \end{aligned}\\
 \leq &\mu (J_{n_{j}+{r-2}}(x))\\ &  \min \Big\{1, c_{0}d       \frac{e^{{2\tau(z_{j})S_{l_{j}}h(z_{j})}}}{   (\gamma_{0}\cdots \gamma_{r-2})  ^{2\tau(z_{j})S_{l_{j}}h(z_{j})} }    q^{2}_{n_{j}+{r-2}} (x)                       \frac{(\gamma_{0}\cdots\gamma_{r-2})^{\tau(z_{j})S_{l_{j}}h(z_{j})} }    {  {e}^{\tau(z_{j})S_{l_{j}}h(z_{j})} }                                        \Big\} \\
 \leq & 6\cdot 4^{r}|J_{n_{j}+{r-2}}(x)|^    \frac{s^{(r)}_{M}-3\epsilon}{1+\epsilon}\min \Big\{1,c_{0}d          \frac{e^{{\tau(z_{j})S_{l_{j}}h(z_{j})}}}{   (\gamma_{0}\cdots \gamma_{r-2})  ^{\tau(z_{j})S_{l_{j}}h(z_{j})} }      q^{2}_{n_{j}+{r-2}} (x)                                \Big\} \\
 \leq & c\Big(\frac{   (\gamma_{0}\cdots \gamma_{r-2})  ^{\tau(z_{j})S_{l_{j}}h(z_{j})}        }{   e^{\tau(z_{j})S_{l_{j}}h(z_{j})}q^{2}_{n_{j}+{r-2}}                       }\Big)^   \frac{s^{(r)}_{M}-3\epsilon}{1+\epsilon}  \\ & \qquad \qquad \cdot \Big( c_{0}d    \frac{e^{{\tau(z_{j})S_{l_{j}}h(z_{j})}}}{   (\gamma_{0}\cdots \gamma_{r-2})  ^{\tau(z_{j})S_{l_{j}}h(z_{j})} }      q^{2}_{n_{j}+{r-2}}(x)                      \Big) ^  \frac{s^{(r)}_{M}-3\epsilon}{1+\epsilon}   \\          
 \leq & cc_{0}d^{ \frac{s^{(r)}_{M}-3\epsilon}{1+\epsilon} }.
\end{align*}
}
By combing all the above cases and using the mass distribution principle, we conclude that 
\begin {equation*}\dim _{\mathrm{H}} \mathcal R_{r}(\tau)\geq   \dim _{\mathrm{H}}   \mathcal E_{\infty} \geq s_{0}={ \frac{s^{(r)}_{M}-3\epsilon}{1+\epsilon}. }\end{equation*}
Since $\epsilon>0$ is arbitrary, as $\epsilon \to 0$ we have $s_{0} \to s^{(r)}_{M}.$ Further letting $M \to \infty,$ 
\begin{equation*}\dim _{\mathrm{H}} \mathcal R_{r}(\tau) \geq s^{(r)}_{\N}. \end{equation*}
 This completes the proof for the lower bound of Theorem \protect\ref{ddd}.
 
  \subsection*{Acknowledgements} 
 The author would like to thank  Prof. Brian A. Davey and Dr. Mumtaz Hussain for useful comments and discussion on this problem. The author is grateful to the anonymous referee 
 for meticulous reading of the paper and for making many helpful suggestions that improved the presentation of this paper.

\def\cprime{$'$} \def\cprime{$'$} \def\cprime{$'$} \def\cprime{$'$}
  \def\cprime{$'$} \def\cprime{$'$} \def\cprime{$'$} \def\cprime{$'$}
  \def\cprime{$'$} \def\cprime{$'$} \def\cprime{$'$} \def\cprime{$'$}
  \def\cprime{$'$} \def\cprime{$'$} \def\cprime{$'$} \def\cprime{$'$}
  \def\cprime{$'$} \def\cprime{$'$} \def\cprime{$'$} \def\cprime{$'$}
  \def\cprime{$'$}

 \bibliographystyle{line}

\end{document}